\documentclass[11pt]{amsart}\usepackage{lmodern}
\usepackage{amsmath}
\usepackage[T1]{fontenc}
\usepackage{amssymb,amscd,latexsym,epsfig,hyperref}
\usepackage[mathscr]{euscript}
\usepackage{graphicx}
\usepackage{tikz-cd}
\usepackage{amsthm}
\usepackage{xcolor}
\usepackage[capitalize]{cleveref}

\DeclareFontFamily{OT1}{rsfs}{}
\DeclareFontShape{OT1}{rsfs}{n}{it}{<-> rsfs10}{}
\DeclareMathAlphabet{\curly}{OT1}{rsfs}{n}{it}

\makeatletter
\newcommand{\eqnum}{\refstepcounter{equation}\textup{\tagform@{\theequation}}}
\makeatother

\makeatletter \@addtoreset{equation}{section} \makeatother

\renewcommand{\theequation}{\thesection.\arabic{equation}}
\numberwithin{equation}{subsection}
\newtheorem{thm}[equation]{Theorem}
\theoremstyle{definition}
\newtheorem{dfn}[equation]{Definition}
\newtheorem{lem}[equation]{Lemma}
\newtheorem{cor}[equation]{Corollary}
\newtheorem{prop}[equation]{Proposition}
\newtheorem{rmk}[equation]{Remark}
\newtheorem{ex}[equation]{Example}

\newcommand\rurl[1]{
	\href{http://#1}{\nolinkurl{#1}}
}

\newtheorem{claim}[equation]{Claim}

\newcommand{\E}{\mathcal{E}}
\newcommand{\coker}{\mathrm{coker}}
\newcommand{\EE}{\curly{E}}
\newcommand{\tr}{\mathrm{tr}}
\newcommand{\rk}{\mathrm{rank}}

\newcommand{\T}{\mathrm{T}}

\newcommand{\id}{\mathrm{id}}

\newcommand{\FF}{\curly{F}}

\newcommand{\LLL}{\curly{L}}
\newcommand{\OO}{\mathcal{O}}
\newcommand{\End}{\mathrm{End}}
\newcommand{\Hom}{\mathrm{Hom}}
\newcommand{\Homm}{\mathcal{H}\! \mathit{om}}
\newcommand{\RHomm}{\mathbf{R}\mathcal{H}om}
\newcommand{\Extt}{\mathcal{E} \! \mathit{xt}}
\newcommand{\Ext}{\mathrm{Ext}}
\newcommand{\ext}{\mathrm{ext}}
\newcommand{\vd}{\mathrm{vd}}
\newcommand{\pot}{\mathbf{R}\Homm_{p_X}(\EE,\EE) }

\newcommand{\Y}{\mathrm{Y}}
\newcommand{\At}{\mathrm{At}}
\newcommand{\G}{\mathrm{G}}
\newcommand{\X}{\mathrm{X}}
\newcommand{\A}{\mathrm{A}}
\newcommand{\N}{\mathcal{N}}

\newcommand{\M}{\mathcal{M}}

\begin{document}
		\title{A perfect obstruction theory for $\mathbf{SU}(2)$-Higgs sheaves}
	\author{Simon Schirren\\
		}
	\maketitle
	\tableofcontents

	\textbf{Abstract.} 
    We present a new method for constructing virtual cycles for rank-2 Higgs sheaves $(E,\phi)$ on a smooth projective surface $S$. Using this, we redefine the $\mathbf{SU}(2)$-perfect obstruction theory previously constructed by Tanaka-Thomas. The key step in our construction involves modifying the $\mathbf{C}^\times$-localisation formula of Graber-Pandharipande by replacing the torus action with an involution $(E,\phi) \mapsto (E^*,-\phi^*)$. 
		\subsection{Summary}
		The first part of the introduction is aimed at non-experts and explains the spectral construction and the idea of virtual cycles, while the second part gives an outlook of this article and provides some background. For the more experienced reader, we recommend to jump directly to \ref{VWinvariants} and read the more condensed summary of Higgs pairs and their spectral sheaves in the following Sec. \ref{preliminaries}. 
		\subsection{Spectral sheaves}
		Let $E$ be a finite dimensional vector space over $\mathbf{C}$ and $\phi \in \End(E)$.\\
		We make the $\mathbf{C}$-module $E$ into a $\mathbf{C}[t]$-module by defining $t.e:= \phi(e)$, which is indeed a commutative action, as $\phi$ commutes with itself and thus defines a sheaf $\E_\phi$ on $\mathbf{A}^1_\mathbf{C}:= \mathbf{Spec}(\mathbf{C}[t])$. \\
		We see that $\E_\phi$ is a torsion sheaf, whose scheme-theoretic support is given by  the zeros of $m_\phi (t)\in \mathbf{C}[t]$, the minimal polynomial of $\phi$. \\
		Conversely, given a coherent torsion sheaf $\E$ on $\mathbf{A}_\mathbf{C}^1$ we set $E:=\pi_*\E=H^0(\E)$ for $ \pi: \mathbf{A}_\mathbf{C}^1 \rightarrow *$  and define $\phi:=\pi_*(t \cdot \id)$ for $t\cdot \id$  the tautological endomorphism acting on $\E$ over $\mathbf{A}_\mathbf{C}^1$. \\
		These two constructions are inverse to each other, so giving a torsion sheaf $\E$ on $\mathbf{A}^1$ is equivalent to giving a pair $(E,\phi)$ of vector space and endomorphism.\\
		$\E_\phi$ admits a natural resolution $$0\rightarrow  \pi^*E \xrightarrow{\pi^*\phi-t\cdot \id} \pi^*E \rightarrow \E_\phi \rightarrow 0$$
		which one could use as a definition $\E_\phi:=\coker(\pi^*\phi-t\cdot \id)$. \\
		$\phi$ is diagonalisable if and only if $m_\phi(t)$ has distinct linear factors. In this case, it agrees with $\chi_\phi(t)$, the characteristic equation of $\phi$ and $\E$ is supported on the eigenvalues of $\phi$. For arbitrary $\phi$, the above $2$-step resolution defines a \textit{divisor class}\footnote{We refer to \cite[Ch.2]{KM} for the definition of the divisor class associated to a $2$-step resolution.} $\textrm{div}(\E_\phi):=\det(\pi^*\phi-t\cdot \id)=\chi_\phi(t)$ for $\E_\phi$.  One can think of $\E_\phi$ as the sheaf of (generalised) eigenspaces of $\phi$, supported on their respective eigenvalues in $\mathbf{C}$. We call $\E_\phi$ the \textbf{spectral sheaf} of the pair $(E,\phi)$. 
		
		\subsection{Higgs pairs}
		We do this now in families: fix a smooth projective surface $S$ and a pair $(E,\phi),$ where $E$ is torsion free rank $r$ sheaf on $S$ and $\phi \in \Hom(E,E\otimes K_S)$  an endomorphism of $E$, twisted by the canonical bundle $K_S$ of $S$. We call such a pair a $K_S$-Higgs pair. The total space of the line bundle $X:=\mathrm{Tot}(K_S) \xrightarrow{\pi} S$ is a non-compact Calabi-Yau-threefold, so its structure map $\pi$ has fibres $\pi^{-1}\{s\}\cong \mathbf{A}_\mathbf{C}^1$ . Restricting $(E,\phi)$ to a point $s\in S$ brings us back to a pair $(E_s,\phi_s)$ of vector space and endomorphism, to which we assign the spectral sheaf $\E$ on the $\mathbf{A}_\mathbf{C}^1$-fibre over $s$. \\
		Varying over all of $S$, this defines a torsion sheaf $\E_\phi $ on $X$, supported on the eigenvalues of $\phi$.
		To see that this local picture glues to a global construction on $X$, we refer to the Details in Sec. \ref{Higgs+spectral}, including the original reference from \cite{TT}.
		\subsection{Moduli spaces}
		One is interested in describing \textit{families} of sheaves $\E$ after fixing rank and Chern classes, for which we need a \textit{moduli space} $\N$, i.e. a scheme whose closed points parametrize isomorphism classes $[\E]$ of those sheaves (the ones of our interest are called stable sheaves, in particular we have $\mathbf{Aut}(\E)=\mathbf{C}^\times$). These spaces $\N$ turn out to have a dimension differing from a given $\mathit{expected \; dimension}$: namely, for a moduli space $\N$, we can associate to one of its closed points $[\E]$ the spaces $\mathrm{def}(\E)$ and $\mathrm{ob}(\E)$, the deformations and obstructions of $\E$ in $\N$. These finite $\mathbf{C}$-vector spaces are given by the cohomology groups $\Ext^i(\E,\E)$ for $i=1,2$. We call the difference
		$$\dim_\mathbf{C}\mathrm{def}(\E)-\dim_\mathbf{C}\mathrm{ob}(\E)=\mathrm{ext^1}-\mathrm{ext^2}$$ the \textit{virtual} or \textit{expected dimension}  $vd$ of $\N$, provided the number is constant over all of $\N$.
		The challenge lies now in defining the \textit{virtual} fundamental class $[\N]^{vir} \in A_{vd}(\N)$ in the Chow group $A_*(\N)$ of the given virtual dimension. This has been introduced in \cite{LT} and \cite{BF} in the late 90's.\\
		In general for a moduli space $\N$ parametrizing stable sheaves $\E$, there are bounds \footnote{See \cite[pp.113-115]{HL}  for a precise statement.}
		$$ \ext^1(\E,\E) \geq \dim_{[\E]}\N  \geq \ext^1(\E,\E)-\ext^2(\E,\E)$$
		and we see $\ext^2=0$  implies the smoothness of $\N$ at $[\E]$. This happens e.g. for the moduli space $\mathcal{M}_L$ of stable vector bundles $E$ and fixed determinant $\det(E)=L$ over a smooth projective curve $C$.  Here, deformations and obstructions are given by $\ext^1(E,E)_0$ and $\ext^2(E,E)_0=0$; the suffix $0$ denotes trace-free. In this case, a natural choice for the "virtual" fundamental class $[\mathcal{M}_L]^{vir}$ would be the usual fundamental class $[\mathcal{M}_L]\in A_*(\mathcal{M}_L)$. \\
		If $\E$ is a stable sheaf on a smooth projective threefold $X$ with $\Ext^3(\E,\E)\cong H^3(\OO_X)$, we have by the Riemann-Roch formula $$\ext^1(\E,\E) - \ext^2(\E,\E)=1+h^3(\OO_X)-\int_X\mathrm{ch}(\E^\vee)\mathrm{ch(\E)}\mathrm{Td}(X),$$
		which only depends on the Chern classes of $\E$ and on $X$, so the RHS is constant over all of $\N$.
		Furthermore, if $X$ is of Calabi-Yau type, Serre duality tells us that 
		$$\ext^1-\ext^2=0.$$
		This means if $[\N]^{vir}$ exists and the moduli $\N$ is compact, one could integrate $$\int_{[\N]^{vir}} 1 \in \mathbf{Z}$$ to an integer. This is called a \textit{virtual count} or \textit{Donaldson-Thomas type} invariant.
		\subsection{Virtual cycles}
		For a fixed point $[\E] \in \N$, the space $\Ext^1(\E,\E)$ equals the tangent space $T_{[\E]}\N$ at $[\E] $. \textit{Kuranishi theory} says that locally analytically, the moduli space $\N$ looks like $K^{-1}(0)$ for a map $$K: \Ext^1(\E,\E)\rightarrow \Ext^2(\E,\E), $$
		Globally, one would hope for a model 
		\begin{equation}\label{Kuranishi}
			\begin{tikzcd}
			&											V \arrow[d]\\
			\mathcal{N} \arrow[r,hook]\arrow[ru,"0"] & \mathcal{A} \arrow[bend right =30,swap]{u}{s}\\
		\end{tikzcd}
		\end{equation}
	where $\mathcal{A}$ is a smooth ambient space for $\N$ and $V$ a vector bundle over $\mathcal{A}$ with a section $s: \mathcal{A} \rightarrow V$ cutting out $\N$, so $s$ defines an ideal sheaf $\mathcal{I}\subset \OO_\mathcal{A}$ such that $\OO_\mathcal{A}/ \mathcal{I}\cong\OO_\N.$\\
	This gives an induced map of differentials on $\N$ $$0\rightarrow \T_\N \rightarrow \T_{\mathcal{A}}|_\N \xrightarrow{ds} V|_{\N} \rightarrow \mathrm{Ob}_\N \rightarrow 0,$$
	where we call $\mathrm{Ob}_\N:=\coker(ds)$ the obstruction sheaf of $\N$ with respect to $V$. Thus, there is a 2-term complex of vector bundles ($\mathcal{A}$ is smooth!) $[V_0 \rightarrow V_1]$ whose cohomology sheaves are tangents and obstructions of $\N$. Dually, this is 
	\begin{center}
		\begin{tikzcd}
			V^*|_{\N} \arrow[r,"ds^*"] \arrow[two heads]{d}& \Omega_\mathcal{A}|_\N  \arrow[equal]{d}\\
			\mathcal{I}/\mathcal{I}^2 \arrow[]{r} & \Omega_\mathcal{A}|_\N.
		\end{tikzcd}
	\end{center}
	Here, the section $s$ induces a map $V^* \rightarrow \OO_\mathcal{A}$ mapping onto $\mathcal{I} \subset \OO_\mathcal{A}$. Its restriction to $\N$ gives the LHS arrow; the lower horizontal map is the exterior derivative. We observe that $\ker(ds^*)$ surjects onto $\ker(\mathcal{I}/\mathcal{I}^2\rightarrow \Omega_{\mathcal{A}}|_\N)$ and that the horizontal arrows have isomorphic co-kernels equal to $\Omega_\mathcal{\N}$. \\
	Thus we see that the global model of Diag. \ref{Kuranishi} induces a map 
	\begin{center}
		\begin{tikzcd}
		V^\bullet \arrow[r,equal] \arrow[d,"\Psi"] & {[V^{-1} \rightarrow V^0]} \arrow[d]\\
		\mathbf{L}_\N \arrow[r,equal] & {[\mathcal{I}/\mathcal{I}^2 \rightarrow \Omega_{\mathcal{A}}|_\N]}
		\end{tikzcd}
	\end{center}
in $\mathbf{D}^b(\N)$, the bounded derived category of coherent sheaves on $\N$, such that $\Psi$ is an isomorphism in degree $0$ and a surjection in degree $-1$. Here, $V^\bullet$ is a 2-term complex of vector bundles $V^i$, which we call the \textit{virtual} cotangent bundle. $\mathbf{L}_\N=[\mathcal{I}/\mathcal{I}^2 \rightarrow \Omega_{\mathcal{A}}|_\N]$ is Illusie's truncated cotangent complex, which is, as an object of $\mathbf{D}^b(\N)$, independent of the choice of embedding $\N \subset \mathcal{A}$. We call such a pair $(V^\bullet, \Psi)$ a perfect obstruction theory on $\N$. \\
Although such a global Kuranishi model is rare, it always exists \textit{locally} (as $\N$ is locally $\mathbf{Spec}(A)$ for a finitely generated $\mathbf{C}$-algebra $A$, defined by some ideal $I \subset \mathbf{C}[x_1,\cdots,x_n]$, so choosing $(V,s)$ amounts to choosing generators for $I$ in a trivial bundle $V$ over $\mathbf{A}^n$).\\
One would then like to define the virtual cycle as $$ [\N]^{vir} = s_0^!(\mathrm{C})$$ for $\mathrm{C}:=\lim_{t\rightarrow \infty}\Gamma_{t\cdot s} \subset V_1=V^{-1,\vee}$. Here $\Gamma_{t\cdot s}$ denotes the graph of the section $t\cdot s$ for $t$ a scalar. This limit can be seen as the section $s$ "made vertical": we refer to \cite[pp.87-88]{F} for an explanation of Fulton-Mac-Pherson's deformation to the normal cone and a graphic. Here, $s_0^!$ is the Gysin map\footnote{See  \cite[pp.112-114]{F} for a definition of the Gysin homomorphism.} $$s_0^!: A_*(V_1)\rightarrow A_{*-\rk(V_1)}(\N)$$ for the zero section $s_0: \N \rightarrow V_1$. 
Although a global Kuranishi model might not exist, \cite{BF} show that given a perfect obstruction theory $(V^\bullet,\psi)$, the cones $\mathrm{C}$ of the local charts glue to a cone $ \mathrm{C} \subset V_1$ and $\N$ inherits a virtual cycle $[\N]^{vir}=s_{0}^{!}[\mathrm{C}]$ of the given expected dimension.\\
 In practice, the cohomology of $V^{\bullet,\vee}$ computes $h^0(V^{\bullet,\vee}|_{[\E]})=\mathrm{def}(\E)$ and  $h^1(V^{\bullet,\vee}|_{[\E]})=\mathrm{ob}(\E)$ for a closed point $[\E]\in \N$. \\
The existence of perfect obstruction complexes giving virtual cycles for fine moduli of stable $E$ of $\rk(E)>0$ under the hypothesis $$\Ext^{i}(E,E)_0=0 , \hspace{10pt} i\neq 1,2$$ is proved in \cite[Cor.4.3]{HT} \\
In this paper, we are interested in the $\rk$ 0 spectral sheaves $\E$ on $X=\mathrm{Tot}(K_S)$, where the perfect obstruction theory is given by the truncated Atiyah class on $\N$ (see \cite[Sec.4.4]{HT})
$$ \tau^{[-1,0]}\RHomm_{p_X}(\EE,\EE) [2] \rightarrow \mathbf{L}_\N$$
after fixing a $2$-term complex $V^\bullet$ representing the LHS. This will be reviewed in Sec. \ref{pots}.
\subsection{Vafa-Witten invariants}\label{VWinvariants}
In this paper, we are interested in the moduli space $\N$ of Higgs pairs $(E,\phi)$ on $S$ or equivalently their spectral sheaves $\E_\phi$ on $X=K_S$. To compute interesting invariants and by what we have discussed before, we replace the fundamental cycle $[\N]$ by a virtual cycle $[\N]^{vir}$ of some given expected dimension. \\
However, $\N$ is non-compact, as it admits a $\mathbf{C}^\times$-action by scaling the Higgs pairs $\lambda: (E,\phi)\mapsto (E,\lambda\phi)$ or equivalently the fibres of $X\rightarrow S$ and thus acts on $\E_\phi$ by pull-back. Thus, a virtual cycle on $\N$ is a priori not interesting; the relevant virtual cycle arises after passing to the $\mathbf{C}^\times$-fixed locus $\N^{\mathbf{C}^\times}$. The "$\mathbf{C}^\times$-localisation of virtual cycles", i.e. localising the virtual cotangent bundle $V^\bullet$ on $\N$ at the fixed locus $\N^{\mathbf{C}^\times}$ goes back to \cite{GP} and allows us to compute invariants in the $\mathbf{C}^\times$-equivariant setting via a (virtual) Bott residue formula $$\int_{[\N^{\mathbf{C}^\times}]^{vir}}\frac{1}{e(N^{vir})}.$$
Here, $e(N^{vir})$ denotes the Euler or top Chern class of the \textit{virtual} normal bundle of the obstruction theory: namely, if $V^\bullet$ is the ($\mathbf{C}^\times$-equivariant) obstruction complex, $N^{vir}$ is constructed by taking the restriction $V^\bullet|_{\N^{\mathbf{C}^\times}}$ to the fixed locus and taking the non-zero weight part. 
The fixed locus $\N^{\mathbf{C}^\times}$ has two components, the "\textit{instanton branch}" $\phi=0$ and the "\textit{monopol branch}" $\phi \neq 0$.\\
This would be a first try to define what is a Vafa-Witten invariant \cite{VW} \footnote{Solutions of the "Vafa-Witten" equations correspond to certain stable holomorphic Higgs pairs $(E,\phi)$, which are expected to have modular properties.} on $\N$
$$\mathsf{VW}_\N:= \int_{[\N^{\mathbf{C}^\times}]^{vir}}\frac{1}{e(N^{vir})}.$$
However, this invariant is zero unless $h^{0,1}(S)=0=h^{0,2}(S)$, which in some sense has to do with the fact that the obstruction complex $V$ giving $[\N^{\mathbf{C}^\times}]^{vir}$ is not fixing $\det(E)$. 
The right way to go around this is going from $\mathbf{U}(r)$-pairs $(E,\phi)$ to $\mathbf{SU}(r)$-pairs, which are defined as  
$$\N^\perp= \{(E,\phi)\in \N: \det(E)\cong \OO_S, \tr(\phi)=0 \}\subset \N $$
before doing the $\mathbf{C}^\times$-localisation. \\
This gives a better virtual cycle $[\N^{\perp,\mathbf{C}^\times}]^{vir}$ and a more sensible Vafa-Witten invariant

$$\mathsf{VW}_{\N^\perp}:= \int_{[\N^{\perp,\mathbf{C}^\times}]^{vir}}\frac{1}{e(N^{vir})}$$
Under the assumption that $S$ is simply connected, we refer to \cite{GSY}.\\
One can show that for $\deg K_S <0$ or $K_S\cong \OO_S$ the $\mathsf{VW}$ invariant for the instanton branch $\mathcal{M}$
is $\mathsf{VW}=\int_{\mathcal{M}}c_{vd}(\Omega_\mathcal{M})=(-1)^{vd}e(\mathcal{M})$ simply the signed topological Euler characteristic of the smooth space $\mathcal{M}$ (see \cite[Sec.7]{TT}). \\
The more interesting monopole branch is discussed in \cite[Sec.8]{TT} under certain assumptions and identifies one of its components with nested Hilbert schemes on $S$.
\subsection{Notation}
Before stating the goal of this paper, we fix some notation. Details to the introduced objects are discussed in Sec. \ref{Higgs+spectral} - \ref{TTtriangle}. \\
We fix once and for all $\mathbf{C}$ as the ground field.
Let $(S,\OO(1))$ be a smooth projective surface with polarisation $\OO(1)$ and $K_S$ its canonical sheaf. We denote by $ X:=\mathrm{Tot}(K_S)\xrightarrow{\pi} S$ the total space of this line bundle with canonical projection $\pi$ to the base. \\
A Higgs pair $(E,\phi)$ is a torsion free coherent sheaf $E$ on $S$ of $\rk(E)=r>0$ together with a section $\phi \in \mathrm{Hom}(E,E\otimes K_S)$. We call $E$ a \textit{Higgs bundle} if it is locally free, otherwise a \textit{Higgs sheaf}. By the spectral construction (see Sec. \ref{Higgs+spectral}), this data is equivalent to a sheaf $\E_\phi=\E$ of dimension $2$ on $X$, where we omit the suffix $\phi$ whenever there is no chance of ambiguity. We call $\E_\phi$ the spectral sheaf corresponding to $(E,\phi)$ and we may write $(E,\phi)=\E_\phi$. \\
All sheaves considered will be (Gieseker)-stable (see Sec. \ref{Gieseker}), in particular we have $\mathrm{Aut}(\E)=\mathrm{Aut}(E,\phi)=\mathbf{C}^\times$. 
$\N$ is a coarse moduli space with fixed invariants for Higgs sheaves on $S$ or equivalently, spectral sheaves on $X$. \\
There is a moduli stack $\M$ of sheaves $E=\pi_*\E$ on $S$ with these invariants. \\
We denote by $\EE$ a (twisted) universal family over $X\times\N$ of spectral sheaves $\E_\phi$ on $X$.
Equivalently, there is a twisted universal Higgs pair $(\mathsf{E},\Phi)$ over $S\times \N$. We give more details to these objects in Sec. \ref{modulispaces} and \ref{universalsheaves}.\\
We denote by $\RHomm(\EE,\EE) \in \mathbf{D}^b(\X\times\N)$, the derived $\Homm$-functor in the bounded category of coherent sheaves on $X\times\N$. We set $$\mathbf{R}\mathcal{H}om_{p_X}(\EE,\EE):= \mathbf{R}{p_{X,*}}(\mathbf{R}\Homm(\EE,\EE))\in \mathbf{D}^b(\N),$$
where $p_X$ is the canonical projection $ X\times \N \rightarrow \N$. Analogously, we define $p_S: S\times \N \rightarrow \N$. We will use for the canonical map $\pi: X\rightarrow S$ and its base change $\pi:=\pi \times \id: X\times \N \rightarrow S \times \N$ the same notation. Furthermore, the derived sheaves $\mathbf{R}p_{S,*}\mathcal{O}_{S}$ and $\mathbf{R}p_{S,*}K_{S}$ denote the push-downs to $\N$ of the sheaves $\OO_{S\times \N}$ and $ K_S\otimes \OO_{S\times \N} $ on $S\times \N$ respectively. \\
There are splittings (Sec. \ref{splitting}) of these sheaves denoted as $$\mathbf{R}\mathcal{H}om_{p_X}(\EE,\EE)= \mathbf{R}\mathcal{H}om_{p_X}(\EE,\EE)_0\oplus\mathbf{R}{p_{S,*}}\OO_S,$$  $$\mathbf{R}\mathcal{H}om_{p_X}(\EE,\EE)= \mathbf{R}\mathcal{H}om_{p_X}(\EE,\EE)^0\oplus \mathbf{R}p_{S,*}K_S[1] $$ 
and $$ \mathbf{R}\mathcal{H}om_{p_X}(\EE,\EE)=\mathbf{R}\mathcal{H}om_{p_X}(\EE,\EE)^\perp \oplus \mathbf{R}p_{S,*}K_{S}[-1]\oplus \mathbf{R}p_{S,*}\mathcal{O}_{S}.$$ 
Here, $\mathbf{R}\mathcal{H}om_{p_X}(\EE,\EE)_0$ denotes trace-free homomorphisms.

\subsection{Goal}
Defining a perfect obstruction theory for $\N^\perp$ takes over thirty pages in \cite{TT}. The aim of this article is to \textit{redefine} this perfect obstruction theory for $\mathbf{SU}(2)$-pairs by identifying them as fixed points  in $ \N$ (we focus entirely on $\rk(E)=2$): Applying the ideas from the $\mathbf{C}^\times$-localisation of Graber-Phandaripande, we replace the torus action by an involution $\iota: \N \rightarrow \N$ generically defined as $$\iota: (E,\phi) \mapsto (E^*,-\phi^*)$$
and one sees quite easily that the fixed locus $\N^\iota$ contains $\N^\perp$ as a connected component. \\
Taking the perfect obstruction theory $V^\bullet$ on $\N$ and making it $\iota$ -equivariant splits, just like in the case of \cite{GP}, $$V^\bullet=V^{\bullet, \iota}\oplus V^{\bullet, mov}$$ into fixed and moving part over $\N^\perp$. As $\iota$ has square equal to the identity, the moving part $V^{\bullet, mov}$ or \textit{virtual co-normal bundle } ( which is the non-zero weight part in \cite{GP}) is the $-1$ eigensheaf of the action of $\iota$ on $V^\bullet$. This is shown in Sec.\ref{determinant} and \ref{trace}.\\
Tanaka-Thomas show for arbitrary $\rk(E)$ that the differentials on tangent complexes of the classifying maps
\begin{equation*}
\begin{tikzcd}
\N \arrow[r] \arrow{d}{\det\pi_*} \arrow{r}{\tr}& \Gamma(K_S)\\
\mathbf{Pic}(S)& 
\end{tikzcd}
\end{equation*}
where $\tr:\E_\phi \mapsto \tr(\phi)$ and $\det:\E_\phi \mapsto \det\pi_*(\E_\phi)$,
commute with the following split-maps of \textit{virtual} tangent complexes via Atiyah classes:
\begin{equation}
\begin{tikzcd}
\mathbf{R}\Homm_{p_{X}}(\EE,\EE)[1] \arrow[r] \arrow[d] & \mathbf{R}p_{S_*}K_S\\
\mathbf{R}p_{S_*}\mathcal{O}_{S}[1] .
\end{tikzcd}
\end{equation}
Our approach is lifting the involution $\iota$ to the tangent/obstruction complex $ \mathbf{R}\mathcal{H}om_{p_X}(\EE,\EE)[1]$, namely a map $$\theta_{\iota}: \mathbf{R}\Homm_{p_{X}}(\EE,\EE)[1]\rightarrow \mathbf{R}\Homm_{p_{X}}(\iota\EE,\iota\EE)[1]$$ whose restriction to $\N^\perp$ gives a genuine endomorphism on $\mathbf{R}\Homm_{p_{X}}(\EE,\EE)[1]|_{\N^\perp}$ with square equal to the identity, such that $$(\mathbf{R}p_{S,*}\OO_S[1]\oplus \mathbf{R}p_{X,*}K_S)|_{\N^\perp}$$
is the $-1$ \textit{eigensheaf} of $\theta_{\iota}$ represented by the complex $V^{\bullet, mov}$. \\
The virtual cycle of $\N^\perp$ is then constructed by taking $V^\bullet$ over $\N^\perp$ and remove the direct summand $V^{\bullet, mov}$. We will show that the fixed part $V^\bullet |_{\N^\perp}^\iota$ representing  $\mathbf{R}\Homm_{p_{X}}(\EE,\EE)[1]|_{\N^\perp}^{\iota}$ defines a perfect obstruction theory for $\N^\perp$. \\
The challenges here were 
\begin{itemize}
	\item to generalise $(E,\phi) \mapsto (E^*,-\phi^*)$ to torsion free sheaves and phrase this in terms of their spectral sheaves on $X$
	\item to define the correct lift $\theta_{\iota}$, compatible with Atiyah classes. For the deformation theory of $\tr(\phi)$, we had to use one result of \cite{TT}. 
\end{itemize}
We hope to construct with the methods presented new Vafa-Witten invariants for different Lie algebras (see \cite{Hi}).

\subsection{Acknowledgements}
I am deeply grateful to my advisor Richard Thomas for his help, his energy and his patience. I thank him for hours and hours of explanation and countless emails. Special thanks to Filippo Viviani, who helped me a lot in the beginning.\\
Furthermore, I thank Georg Oberdieck and Woonam Lim for extremely helpful conversations and Tim Bülles, Denis Nesterov, Andrea Ricolfi and Sandro Verra for inspiring discussions. \\
Special thanks go to Martijn Kool for suggestions and helpful comments after the first submission. In addition, we would like to thank Daniel Huybrechts for a conversation about expressing the $\mathbf{SU}(2)$-locus as fixed points. \\
Lastly, I want to thank the anonymous referees for their remarks and suggestions.
\section{Preliminaries}\label{preliminaries}
\subsection{Equivariant sheaves}
Let $\G$ be an algebraic group acting on a quasi-projective scheme $Y$. 
For all $g\in \G$, the pullback $g^*:\FF \mapsto g^*\FF$ acts as an auto-equivalence on $\mathbf{D}^b(Y)$ with inverse $g_*=g^{-1,*}$. \\
We define for the special case $\G\cong \mathbf{Z}/2\mathbf{Z}$, i.e. $\G= \langle \iota \rangle $ given by an involution $\iota: Y\rightarrow Y$ the $\iota$-\textit{equivariant} derived category of coherent sheaves $$\mathbf{D}^b(Y)^{\langle \iota \rangle}$$  as follows:
\begin{dfn}\label{equi}
	We call an object $\FF\in \mathbf{D}^b(Y)$ $\iota$-\textit{equivariant} if there is an isomorphism $\theta_{\iota}:\FF \rightarrow \iota^*\FF$ such that 
	\begin{equation} \label{triangle}
	\begin{tikzcd}
	\FF\arrow[r,"\theta_{\iota}"] \arrow[dr,equal]& \iota^*\FF \arrow[d,"\iota^*\theta_\iota"] \\
	&(\iota^*)^2\FF
	\end{tikzcd}
	\end{equation}
	commutes. 
\end{dfn}
\begin{dfn}
	We call a morphism $\psi:\FF \rightarrow \FF'$ in $\mathbf{D}^b(Y)$ a morphism of $\iota$-equivariant sheaves, if for $\FF,\FF'$ $\iota$-equivariant as above, the triangles induced by $\theta_\iota: \FF \rightarrow \iota^*\FF$ and  $\theta_\iota': \FF' \rightarrow \iota^*\FF' $ map to each other via
	\begin{equation} \label{equivmor}
	\begin{tikzcd}
	\FF \arrow[r,"\theta_{\iota}"] \arrow[d,"\psi"]& \iota^*\FF \arrow[d,"\iota^*\psi"]\\
	\FF' \arrow[r,"\theta_\iota'"]& \iota^*\FF'\\
	\end{tikzcd}
	\end{equation}
	and its pullback by $\iota$.
\end{dfn}
\begin{dfn}
	Pairs $(\FF, \theta_\iota)$ together with their compatible maps $\psi$ form a category that we denote by $\mathbf{D}^b(Y)^{\langle \iota \rangle}$. There is a forget functor 
	\begin{align*}
	\mathbf{D}^b(Y)^{\langle \iota \rangle} &\rightarrow \mathbf{D}^b(Y),\\
	(\FF,\theta_{\iota})&\mapsto \FF
	\end{align*}
	which forgets the equivariant structure.
\end{dfn}
	For a thorough introduction, we refer the reader to \cite[pp.3-6]{R} for $\G$-equivariant categories of coherent sheaves and \cite[pp.3-10]{BO} for a more general setup. We remark the following:
\begin{rmk}
	As $\langle \iota \rangle $ is finite, the fixed locus $Y^{\iota}\subset Y$ is a (possibly empty) closed subvariety of $Y$. \\
	If $Y=\{*\}$, a $\iota$-equivariant coherent sheaf on $Y$ is just a finite dimensional vector space $V$ with a representation
	\begin{align*}
		\theta: \langle \iota \rangle &\rightarrow \mathrm{Aut}(V)\\
		\iota &\mapsto \theta_{\iota}: v \mapsto \theta_{\iota}(v)
	\end{align*}

	 Here, we make $V$ into a $\mathbf{C}[\langle \iota \rangle]$-module via $\iota.v:=\theta_\iota(v)$ for $\theta_\iota:V\rightarrow g^*V=V$ the linearisation map. As $\theta_{\iota}$ has square equal to the identity, this splits $V\cong V^{+}\oplus V^{-}$ into $\pm1$ eigenspaces, corresponding to $\iota$-fixed and moving part.
\end{rmk}	
\begin{lem}
	Let $V^\bullet$ be a $\iota$-equivariant complex of coherent sheaves on $Y$ as defined in Def. \ref{equi} and assume $Y^{\iota}\neq \emptyset$. Then $V^\bullet|_{Y^\iota}$ carries the structure of a complex of a $\OO_{Y^\iota}[\langle \iota \rangle]$-module.  
\end{lem}	
\begin{proof}
	We first reduce to the local case: Let $Y^\iota=\mathbf{Spec}(A)$ for a finitely generated $\mathbf{C}$-algebra $A$, so we may assume $V^\bullet|_{Y^\iota}$ is given by a complex of finite $A$-modules $M^\bullet$, on which $\theta_{\iota} $ acts. For $v \in M^k$, defining  $\iota.v:=\theta_{\iota}^k(v)$ for $\theta_{\iota}^k:M^{k} \rightarrow M^{k}$ and extending $A$-linearly makes each $M^{k}$ of $M^\bullet$ into an $A[\langle \iota \rangle]$-module, as we have
	\begin{center}
		\begin{tikzcd}
		M^k \arrow[r,"\theta_{\iota}^k"] \arrow[rd,equal]& M^k \arrow[d,"\theta_{\iota}^k"]\\
		&	M^k.
		\end{tikzcd}
	\end{center} \noindent
	 As the $\theta_{\iota}^k$ are compatible with the differential maps of $M^\bullet$, this lifts $M^\bullet$ to a complex of $A[\langle \iota \rangle]$-modules. Using the globally defined equivariant structure given by $\theta_\iota \in \End(V^\bullet|_{Y^\iota})$, this construction glues to all of $Y^\iota$ and makes $V|_{Y^\iota}^\bullet$ into a complex of $\OO_{Y^\iota}[\langle  \iota \rangle]$-modules. 
\end{proof}
Furthermore, for $\G$ a finite group, we have the following lemma:
\begin{lem}\label{Glinear}
	Let $\G$ be a finite group and $V$ a coherent $\OO[\G]$-module. Then the functor of taking fixed parts $V \mapsto V^\G$ is exact. 
\end{lem}
\begin{proof}
	We remark that the functor $V\mapsto V^\G$ is naturally isomorphic to $\Homm_{\OO[\G]}(\OO,\_)$, where $\OO$ is endowed with the trivial $\G$-action. Thus, taking fixed parts is exact if and only if $\OO$ is projective as an $\OO[\G]$-module. \\
	As previously, we discuss the affine case first, i.e. $V$ is the associated sheaf of modules of a finite $A$-module $M$: The natural projection $\G \twoheadrightarrow \{1\}$ has a right-inverse given by $1 \mapsto \frac{1}{|\G|}\sum_{g \in \G}g$, which defines an epimorphism $ A[\G] \twoheadrightarrow A$ with a section $A \hookrightarrow A[\G]$. This shows that $A$ is projective, as it is a direct summand of the free module $A[\G]$. Globally, the sheafification of this construction realises $\OO$ as a direct summand of the free sheaf $\OO[\G]$. 
\end{proof}
\begin{cor}
	We see this easily extends to a complex $V^\bullet$ of $\OO[\G]$-modules. 
\end{cor}
\begin{rmk}
	We end this section with the remark that a $\iota$-equivariant sheaf (or complex thereof) $\FF$ over $Y$ is a $\OO[\iota]$-module over $Y^\iota$, which splits $\FF|_{Y^\iota}\cong\FF^{+}\oplus \FF^{-}$ into invariant and moving part. 
\end{rmk}

\subsection{Equivariant embeddings} \label{equiv}
\begin{dfn}
	We call an invertible sheaf $\mathscr{L} \in \mathbf{D}^b(Y)^{\langle \iota \rangle}$ a $\iota$-linearised line bundle.
\end{dfn}
If $\mathscr{L}$ is a very ample, $\iota$-linearised line bundle on $Y$, the isomorphism $\theta_\iota: \mathscr{L} \xrightarrow{\sim} \iota^*\mathscr{L}$ defines an action of $\iota$ on the vector space $H^0(\mathscr{L})$ via the composition $$H^0(\mathscr{L}) \xrightarrow{\theta_{\iota}} H^0(\iota^*\mathscr{L}) \xrightarrow{\sim} H^0(\mathscr{L}),$$
where the second arrow is the natural pullback. The induced embedding $Y \hookrightarrow \mathbf{P}(H^0(\mathscr{L})^\vee)$ lifts $\iota$ to a projective ambient space $\mathbf{P}$ of $Y$, extending $\langle \iota \rangle$. \\
\begin{lem}\label{equivemb}
Let $Y$ be a quasi-projective variety and $\iota$ an involution acting on $Y$. Then there exists a $\iota$-equivariant embedding into a smooth ambient space $Y \subset \mathcal{A}$.	
\end{lem}	
\begin{proof}
	As $Y$ is quasi-projective, we choose a very ample line bundle $\mathscr{L}$ on $Y$. Then the line bundle  $\mathscr{L} \otimes \iota^* \mathscr{L}$ is again very ample and now $\iota$-linearised by swapping the factors.\\
		As explained above in $\ref{equiv}$, this gives an embedding into a smooth ambient space $\N\hookrightarrow \mathcal{A}$, lifting the action of $\iota$. 
\end{proof}
\subsection{Illusie's cotangent complex} \label{Atiyahclass}
\begin{dfn}\label{represbycomplex}
	We call a morphism $\psi: \FF \rightarrow \FF'$ in $\mathbf{D}^b(Y)$ represented by complexes if we fix representations $\FF^\bullet \rightarrow \FF$ and $\FF^{'\bullet} \rightarrow \FF'$ of coherent sheaves $\FF^{i},\FF^{'i}$ computing the cohomology of $\FF$ and $\FF'$, together with a map of complexes $\psi^\bullet: \FF^\bullet \rightarrow \FF^{'\bullet}$ such that 
	\begin{center}
		\begin{tikzcd}
		\FF^\bullet \arrow[d] \arrow[r,"\psi^\bullet"]& \FF'^\bullet \arrow[d]\\
		\FF \arrow[r,"\psi"]& \FF'		
		\end{tikzcd}
	\end{center}
	commutes.
\end{dfn}
\begin{dfn}
	For a complex $\FF^\bullet \in \mathbf{D}^b(Y)$, we define the canonical truncation $\tau^{\geq a}\FF^\bullet$ to be the complex defined as
    \begin{center}
		\begin{tikzcd}
         \FF^\bullet \arrow[d] & \dots \arrow[r]  & \FF^{a-1} \arrow[d] \arrow[r]& \FF^{a} \arrow[d]\arrow[r] & \FF^{a+1} \arrow[d] \arrow[r] & \dots \\
		\tau^{\geq a} \FF^\bullet & \dots \arrow[r] & 0 \arrow[r]& \coker(d_{n-1}) \arrow[r] & \FF^{a+1} \arrow[r] & \dots 		
		\end{tikzcd}
	\end{center}    
    which implies
	\begin{equation*}
	H^{i}(\tau^{\geq a}\FF^\bullet)=
	\begin{cases}
	H^{i}(\FF^{\bullet}) & \text{if } i  \geq a\\
	0      & \text{else }.
	\end{cases}
	\end{equation*}
    Dually we may define $\tau^{\leq a}\FF^\bullet$ and similarly $\tau^{[a,b]}\FF^\bullet$, the two-sided truncation in $\mathbf{D}^{[a,b]}(Y)$ such that $H^{i}(\tau^{[a,b]}\FF^\bullet)=H^{i}(\FF^\bullet)$ whenenver $i\in [a,b]$.  
\end{dfn}
\begin{dfn}\label{Illusie}
We denote by $\tau^{\geq -1}\mathbf{L}_Y$ \textit{Illusie's truncated cotangent complex} and refer to \cite[pp.160-172]{I} for are detailed introduction. In our case, the following special case is sufficient: Namely, if $Y\subset \mathcal{A}$ admits an embedding into a smooth $\mathcal{A}$, $\tau^{\geq -1}\mathbf{L}_Y$ is represented by the two term complex $$\tau^{\geq 1}\mathbf{L}_Y:= [\mathcal{I}/\mathcal{I}^2 \xrightarrow{d} \Omega_{\mathcal{A}}|_{{Y}}] \in \mathbf{D}^{[-1,0]}(Y)$$ for $\mathcal{I} \subset \mathcal{O}_{\mathcal{A}}$ the ideal sheaf of this embedding; we note that $h^0({\tau^{\geq -1}\mathbf{L}_{Y}})\cong \Omega_Y$. We omit the truncation symbol $\tau^{\geq -1}$ from notation for the rest of this discussion.
\end{dfn}
\begin{rmk}
	It is a well known fact that as an object in $\mathbf{D}^b(\N)$, $\mathbf{L}_Y$ is independent of the choice of $\mathcal{A}$, see eg. \cite[pp.16-17]{R}.\\
	Furthermore, $\mathbf{L}_Y$ is functorial, i.e. for morphisms $f: Y \rightarrow Y'$ there are differentials $$f_*: f^*\mathbf{L}_{Y'} \rightarrow \mathbf{L}_{Y}.$$
	For later computations,we denote by $\mathbf{T}:=\mathbf{L}^\vee$ the tangent complex, dual to $\mathbf{L}$ with covariant differentials $$f_*: \mathbf{T}_{Y} \rightarrow f^*\mathbf{T}_{Y'},$$
	where, by abuse of notation, both maps are denoted as $f_*$.
\end{rmk}
We cite the following observation from \cite[p.17]{R}: 
\begin{rmk}\label{equivillusie}
	If $Y\subset \mathcal{A}$ is a smooth embedding extending any involution $\iota$ as in \ref{equivemb}, then $\mathbf{L}_Y$ is canonically $\iota$-equivariant, i.e. admits a canonically lift to $\mathbf{D}^b(Y)^{\langle i \rangle}$. 
\end{rmk}
\subsection{The Atiyah class}\label{fullAtiyah}

Let $$i_{\Delta_{Y}}: Y \rightarrow \Delta_{Y} \subset Y\times Y$$ be the diagonal map and let $p_1,p_2$ be the canonical projections $\Y\times \Y \rightarrow Y$.\\
The \textit{universal Atiyah class} is given by a morphism $$\alpha_{Y}: \OO_{\Delta_{Y}} \rightarrow i_{\Delta_Y,*}\mathbf{L}_Y[1] \in \mathbf{D}^b(Y\times Y),$$ see  \cite[Ch.5]{HT} for details.\\
For $\FF \in \mathbf{D}^b(Y)$,
the \textit{full Atiyah class} $\At_\FF$ of $\FF$ is $$\mathrm{At}_\FF:=p_{2,*}(p^*_1 \FF \otimes \alpha_{Y}): \FF \rightarrow \FF \otimes \mathbf{L}_{Y}[1], $$
which can be seen as a morphism $$\RHomm(\FF,\FF) \rightarrow \mathbf{L}_Y[1] \in \mathbf{D}^b(Y).$$
\subsubsection{Naturality}
A simple but important observation is the fact that the Atiyah class is natural: Namely, if $\psi: \FF \rightarrow \FF'$ is a morphism in $\mathbf{D}^b(Y)$, then 
\begin{equation}\label{T Atiyah}	
\begin{tikzcd}
\FF \arrow[r,"\psi"] \arrow[d,"\At_{\FF}"]&\FF' \arrow[d,"\At_{\FF'}"]\\
\FF\otimes \mathbf{L}_{Y}[1] \arrow[r,"\psi\otimes 1"] &\FF' \otimes \mathbf{L}_{Y}[1]	
\end{tikzcd}
\end{equation}
commutes, by functoriality of $p_{2,*},p_1^*$ and tensor product.\\
\subsubsection{Functoriality}
Let $f:Y \rightarrow Y$ be an endomorphism. By functoriality of the Atiyah class we mean that the pullback of $\At_{\FF}: \RHomm(\FF,\FF) \rightarrow \mathbf{L}_{Y}[1]$ by $f$ composed with the natural differential $ f^*\mathbf{L}_{Y} \xrightarrow{f_*} \mathbf{L}_{Y}$ equals the Atiyah class of the pullback, i.e. $$\At_{f^*\FF}=f_* \circ f^*\At_\FF.$$
\subsection{Perfect obstruction theories}\label{pots}
Let $\N$ be any moduli space parametrising compactly supported stable sheaves $\E$ on a smooth threefold $X$ with fixed invariants $(\rk(\E),c_1(\E),c_2(\E,),c_3(\E))$.
\subsection{Virtual cycles}
A closed point $[\E]\in \N$ corresponding to a stable sheaf $\E$ on $X$ has deformations given by $\Ext^1(\E,\E)$ and obstructions by $\Ext^2(\E,\E)$. We call the integer $$\vd:= \ext^1(\E,\E)-\ext^2(\E,\E)=1-h^3(\OO_X)-\int_X \mathrm{ch}(\E^\vee)\mathrm{ch}(\E)\mathrm{Td}_X $$ the \textit{virtual} or \textit{expected dimension} of $\N$. This number is indeed constant on $\N$, as the integral on the RHS depends only the Chern classes of $X$ and $\E$. Furthermore, as $\mathrm{supp}(\E)\subset X$ is compact, the integral is well-defined. \\
Globally, we want to find a $2$-term complex of vector bundles on $\N$ $$V^\bullet=[V^{-1}\rightarrow V^{0}]$$  such that $h^0(V^{\bullet,\vee}|_{[\E]})=\Ext^1(\E,\E)$ and $h^{1}(V^{\bullet,\vee}|_{[\E]})=\Ext^2(\E,\E)$.\\
The right data relating extrinsic obstructions (those coming from $V^{\bullet,\vee}$) with intrinsic ones (coming from $\mathbf{L}_\N^\vee$) is a \textit{perfect obstruction theory}.\footnote{We omit the definition of a more general \textit{obstruction theory} which is stated in \cite[Def.4.4]{BF}.}
\begin{dfn}
	A perfect obstruction theory on $\N $ is a pair $(V^\bullet,\psi)$ consisting of
	\begin{itemize}
		\item a $2$-term complex $V^\bullet=[V^{-1}\rightarrow V^0] \in \mathbf{D}^{[-1,0]}(\N)$ of vector bundles
		\item a morphism $\psi: V^\bullet \rightarrow \mathbf{L}_\N$ in $\mathbf{D}^b(\N)$, such that $h^0(\psi)$ is an isomorphism and $h^{-1}(\psi)$ onto. 
	\end{itemize}		
\end{dfn}
\begin{rmk}\label{virtualtangentbundle}
	The complex $V^{\bullet,\vee}$ is often called the \textit{virtual tangent bundle}, denoted by $\mathbf{T}_{\N}^{vir}$. 
	Equipping $\N$ with a perfect obstruction theory $(V^\bullet, \psi)$ is called a \textit{virtually smooth} structure and we remark that virtual smoothness depends on the choice of $V^\bullet$.
\end{rmk}
 The moduli space $\N$ then inherits a  \textit{virtual cycle} $$[\N]^{vir}_{V^\bullet}:=0^{!}_{V_1}[\mathrm{C}]\in A_{\mathrm{vd}}(\N),$$ in the Chow group $A_*(\N)$, where $0^{!}_{V_1}$ is the Gysin map for $V_1=V^{-1,\vee}$ and $\mathrm{C}$ is the Behrend-Fantechi cone, which is a closed subcone of $V_1$, as defined in  \cite[Def.5.2]{BF}. \\ $[\N]_{V^\bullet}^{vir}$ is called the \textit{virtual fundamental class} of $\N$ with respect to $V^\bullet$, which depends on the perfect obstruction theory up to isomorphism in $\mathbf{D}^b(\N)$. 
\begin{rmk}
	If $X$ is a threefold of Calabi-Yau type, then the existence of a perfect obstruction complex $V^\bullet$ equips $\N$ with a virtual cycle of dimension 0, since by Serre-duality on $X$ we have $$\Ext^1(\E,\E)\cong \Ext^2(\E,\E)^*,$$ thus 
	 $\vd=0$. \\
	Such a perfect obstruction theory  is called \textit{symmetric} (i.e. deformations are dual to obstructions). If $\N$ was compact, one could integrate over $[\N]^{vir}$, which would give an actual sheaf count $$\int_{[\N]^{vir}}1 \in \mathbf{Z},$$
	a \textit{virtual Euler characteristic} of $\N$, where the final example below motivates this terminology.
\end{rmk}
\begin{ex}
	Although the moduli space $\N$ of this article is neither smooth nor a local complete intersection, we state a few examples under these assumptions:\\
	Any smooth variety $\N$ admits a natural perfect obstruction theory given by $[0 \rightarrow \Omega_{\N}]$. In that case, $[\N]^{vir}=[\N]$ agrees with the usual fundamental class. Similarly, if $\N$ is a local complete intersection, $\mathbf{L}_\N=[\mathcal{I}/\mathcal{I}^2\rightarrow \Omega_{\mathcal{A}}|_\N]$ is perfect and one can set $V^\bullet:=\mathbf{L}_\N$ with $\psi$ being the identity.\\
	We end this section with an alternative example of a zero-dimensional virtual cycle on a smooth projective $\N$: namely in this case, we may define $V^\bullet :=[\Omega_{\N} \xrightarrow{0} \Omega_{\N}]$ with the obvious map to $\mathbf{L}_\N=\Omega_{\N}$. It can then be shown that $$[\N]_{V^\bullet}^{vir}=e(\Omega_{\N}^\vee)\cap [\N],$$
	which integrates to the holomorphic Euler characteristic  $\chi(\N)\in \mathbf{Z}$.\footnote{See \cite[Prop.7.3]{BF}.} 
\end{ex}

\section{Setup}

\subsection{Setup}
Let $S$ be a smooth projective surface over $\mathbf{C}$ with polarisation $\OO_S(1)$. We denote by $X:= \mathrm{Tot}(K_S) \xrightarrow{\pi} S$ the total space of the canonical sheaf $K_S$ with structure map $\pi$.
\subsection{The category $\mathbf{Higgs}(S)$}
We consider $K_S$-Higgs pairs $(E,\phi)$ where $E$ is a coherent sheaf on $S$ of $\rk(E)=r$ together with a $K_S$-twisted endomorphism $\phi:E\rightarrow E \otimes K_S$. A morphism $f$ between Higgs pairs $(E',\phi')$ and $(E,\phi)$ is a commutative diagram 
\begin{equation}
	\begin{tikzcd}
	E' \arrow[r,"\phi'"] \arrow[d,"f "] & E'\otimes K_S \arrow[d,"f\otimes 1"] \\
	E \arrow[r,"\phi"] & E \otimes K_S.
	\end{tikzcd}
\end{equation}
Taking kernels and co-kernels of the columns defines kernels and co-kernels of Higgs pairs, which form the abelian category $\mathbf{Higgs}(S)$.
\subsection{Higgs bundles and their spectral sheaves}\label{Higgs+spectral}
Instead of working with Higgs pairs $(E,\phi)$, we consider their \textit{spectral sheaves} $\E_\phi$ on $X$, which are built as follows: Over a point $s\in S$, we attach the eigenspaces of $\phi_s$ acting on $E_s$ to their eigenvalues on the fibre $X_s\cong \mathbf{A}^1_\mathbf{C}$.\\
Globally on $S$, we make $E$ into a $\pi_*\OO_X=\oplus_i K_S^{-i}-$module via $$E\otimes K_S^{-i} \xrightarrow{\phi^{i}} E. $$ This gives a torsion sheaf $\E_\phi$ on $X$, preserving stability (definition recalled below) defining an equivalence of categories $$\mathbf{Coh}_c(X)\cong \mathbf{Higgs}(S)$$ (see \cite[Prop.2.2]{TT}) between Higgs pairs on $S$ and coherent sheaves on $X$ of compact support, where the arrow from right to left is the spectral construction. \\
Conversely, starting with a compactly supported coherent sheaf $\E$ on $X$, its push-down $E:=\pi_*\E$ is a torsion free, coherent sheaf and we get $\phi:=\pi_*(\tau\cdot \id)$ from the action of $\tau\cdot \id $ on $X$ for the tautological section $\tau \in \pi^*K_S$.
\subsection{ Gieseker Stability}\label{Gieseker}

A Higgs pair $(E,\phi)$ on $S$ is \textit{Gieseker stable} with respect to $\OO_S(1)$ if  
\begin{equation}\label{stab}
\frac{\chi(F(n))}{\mathrm{rank}(F)} < \frac{\chi(E(n))}{\mathrm{rank}(E)} \; \textrm{for}\;  n\gg 0,
\end{equation} and all $\phi$-invariant proper non-zero subsheaves $F \subset E$.\\
For the stability of $\E_\phi$ on $X$, we introduce the following notation:
\begin{dfn}
	For a spectral sheaf $\E_\phi$ on $X$, we denote by $r(\E_\phi)$ the leading coefficient of the Hilbert polynomial of $\E_\phi$, which agrees with the one of $E$: Indeed, as $\pi_*(\E_\phi(n))= \pi_*\E_\phi\otimes \OO(n)=E(n)$, we have $\chi(\E_\phi(n))=\chi(\pi_*\E_\phi(n))=\chi(E(n))$ and we can write $$r(\E_\phi)=\mathrm{rank}(E)\int_S h^2=\mathrm{rank}(E)\deg(S).$$
\end{dfn}
A Gieseker stable Higgs pair $(E,\phi)$ with respect to $\OO_S(1)$ is equivalent to a Gieseker stable spectral sheaf $\E_\phi$ with respect to the polarisation defined by $\pi^*\OO_S(1)$ on $X$. This is the condition
\begin{equation}
\frac{\chi(\mathcal{F}(n))}{r(\mathcal{F})} < \frac{\chi(\E_\phi(n))}{r (\E_\phi)} \; \textrm{for}\;  n\gg 0,
\end{equation}
for all proper non-zero subsheaves $\mathcal{F} \subset \E_\phi$. Here and for the rest of this discussion, we denote $\OO_X(1):=\pi^*\OO_S(1)$. \\
 \subsection{Chern classes} \label{chern classes}
  The Chern classes for $\E_\phi$ on $X$ are related to those of $(E,\phi)$ on $S$ by  Grothendieck-Riemann-Roch for $(E,\phi=0)$ on $S$, which we identify with a spectral sheaf $\E_{0}=i_*E$, supported on $S$ via push-forward along the zero section $i: S\hookrightarrow X$. \\
 We will restrict entirely to the case $\rk(E)=2$ in this discussion.\\
 Then $\mathrm{ch}(\E_\phi)=i_*(\mathrm{ch}(E_\phi) \cdot \mathrm{td}(T_i))$ where $\mathrm{td}(T_i)=\mathrm{td}(K_S)^{-1}$, which gives \\
 $c_1(\E_\phi)=2[S] $ for the cycle class $ [S] \in H^2(X,\mathbf{Z}),$\\
 $c_2(\E_\phi)=i_*(-3c_1(S)-c_1),$\\
 $c_3(\E_\phi)=i_*(c_1^2-2c_2+3c_1\cdot c_1(S)+4c_1(S)^2)$,\\
 where $c_1(S)=-c_1(K_S)$. 
 \subsection{Slope stability} Using the same notation as above, a Higgs pair $(E,\phi)$ is slope-stable if 
 \begin{equation}
     \mu(F) < \mu(E).
 \end{equation}
 where $\mu(E):=\frac{\deg(E)}{\rk(E)}$, where $\deg(E)$ is the degree of the cycle $c_1(E).H$. We further remark that if denominator and numerator are coprime, then Gieseker-stable $=$ slope stable.
\subsection{Flat families}\label{flatfamilies}
We are interested in moduli problems and \textit{families} of these sheaves and start with recalling the definition of flat families of spectral sheaves $\E_\phi$ on $(X,\OO_X(1))$. 
\begin{rmk}
	We remark first that fixing $ch(\E_\phi)$ as above fixes its Hilbert polynomial  $P(z) \in \mathbf{Q}[z]$ by Riemann-Roch.
\end{rmk}
\begin{dfn}
	For a fixed polynomial $P$ and for $T$ a Notherian separated scheme, a $T$-flat family of spectral sheaves is a coherent sheaf $\tilde{\E}$ on $X\times T$, such that for all $t\in T$, $\tilde{\E}(t):= \tilde{\E}|_{X_t}$ is a stable and flat spectral sheaf on $X$ with Hilbert polynomial $P$. \\
	We call $\tilde{\E}$ an \textit{Artinian family}, if $T=\mathbf{Spec}(A)$ for a local Artinian $\mathbf{C}$-algebra $A$ and use the notation $\tilde{\E}=\E_A$ here. This will be used in Sec. \ref{Artinian}.
\end{dfn} 
\begin{rmk}\label{fixedch}
	 Fixing only the Hilbert polynomial $P$ of $\E_\phi$ admits finitely many choices for $\mathrm{ch}(\E_\phi)$. For the rest of this discussion, we fix $\mathrm{ch}(\E_\phi)$ and thus $P$.
\end{rmk}
\subsection{Moduli spaces}\label{modulispaces}
We start with the definition of the moduli functor $\curly{N}_X^P$.
\begin{dfn}
	We define the functor 
	\begin{align*}
		\curly{N}_X^P:( \mathbf{Sch}/\mathbf{C})^{op} \rightarrow \mathbf{Sets}
	\end{align*}
	as follows: For a scheme $T$, we define
	\begin{align*}
		\curly{N}_X^P(T):=\{& \text{isomorphism classes of}\; T\text{-flat families} \;  \tilde{\E} \\& \; \text{over } X\times T \; \text{such that for all} \; t \in T, \\ &\tilde{\E}(t) \;  \text{is stable with Hilbert polynomial} \; P \}/{\sim},
	\end{align*}
	where $\tilde{\E} \sim \tilde{\E}'$ if and only if $\tilde{\E}\cong \tilde{\E}' \otimes p_T^*L$, for some line bundle $L$.\\
	If $f: T'\rightarrow T$ is a morphism of schemes, we define the pull-back  
	\begin{align*}
			\curly{N}_X^P(f): \curly{N}_X^P(T) &\rightarrow \curly{N}_X^P(T')\\ [\tilde{\E}] &\mapsto [(\id_X\times f)^*\tilde{\E}]
	\end{align*}

\end{dfn}
\begin{dfn}
	We recall from \cite[Sec. 4]{HL} that the functor  $\curly{N}_X^P$ admits a coarse moduli space $\N^P_X$. Only fixing $P$, this space might have several components and we choose the component $\N_X(0,c_1,c_2,c_3)$ corresponding to $ch(\E_\phi)$ of Rmk. \ref{fixedch}, which we will simply abbreviate by $\N$.
\end{dfn}
\begin{rmk}
	$\N$ is a quasi-projective scheme, whose points parametrise isomorphism classes of stable rank-0 sheaves $[\E_\phi]$ on $X$ with the invariants $c_i(\E_\phi)$ of \ref{chern classes}. We remark that the stability of a sheaf $\E_\phi$ implies it is simple, i.e. $\mathrm{Aut}(\E_\phi)\cong \mathbf{C}^\times$. \\
	Alternatively, this is the moduli space $ \N=\N_S(2,c_1,c_2)$ of Higgs sheaves $(E,\phi)$ on $S$. We note that due to the $\mathbf{C}^\times$-action given by $ (E,\phi) \mapsto (E,\lambda \phi)$, $\N$ is non-compact.\footnote{In terms of $\E_\phi$, this is the action coming from scaling the fibres of $X\rightarrow S$.}
\end{rmk}
\subsection{Universal sheaves}\label{universalsheaves}
Although $\N$ is in general not necessarily a fine moduli space and hence does not admit a universal family, it always admits a \textit{twisted} universal family $\EE \; \text{over} \; X\times \N$.
Here, $\EE$ is locally well-defined, but might not glue to a sheaf on all of $X\times \N$, due to $\mathbf{C}^\times$-automorphisms.\footnote{We refer to \cite[p.26-37]{C} for derived categories of twisted families.}
\begin{rmk}\label{classifyingmap}
	We remark that the sheaves we are mostly interested in are the cohomology sheaves $\Extt_{p_X}^{i}(\EE,\EE)$, which always exist globally, independent of any choices, see e.g. \cite[Sec. 10.2]{HL}.\\
	Under the spectral construction, $\EE$ is equivalent to a twisted universal Higgs sheaf $(\mathsf{E},\Phi)$ over $S\times \N$. Here, $\mathsf{E}=\pi_*\EE$ is now a twisted family over $S\times \N$ and $\Phi$ arises as the push-down of $\tau\cdot \id$, where $\tau$ is the tautological section of $\pi^*K_S\otimes \OO$ over $X\times\N$. As $\pi$ is affine, the flatness of $\EE$ implies the one of $\mathsf{E}$. This defines a classifying map "forgetting $\phi$", $$\Pi: \N \rightarrow \M,$$ on closed points given by $[\E_\phi]\mapsto [\pi_*\E_\phi]$ or equivalently $[(E,\phi)] \mapsto [E]$. Here, $\mathcal{M}$ is the moduli stack of coherent sheaves on $S$ with the Chern classes $c_i(E)$ from above. We observe that the fibres of this map are linear, i.e. at $[E] \in \M$ they are given by $\Hom(E,E\otimes K_S)$.
\end{rmk}

\subsection{Deformation theory}\label{TTtriangle}
There is an exact triangle $$\mathbf{R}\mathcal{H}om_{p_X}(\EE,\EE) \xrightarrow{\pi_*} \mathbf{R}\mathcal{H}om_{p_S}(\mathsf{E},\mathsf{E}) \xrightarrow{[\_,\Phi]} \mathbf{R}\mathcal{H}om_{p_S}(\mathsf{E},\mathsf{E}\otimes K_S) \xrightarrow[{[1]}]{\partial}$$
that we briefly explain at single points $[\E_\phi] \in \N$: The first map $\pi_*$ relates deformations of $\E_\phi$ on $X$ with the one of $\pi_*\E_\phi=E$ on $S$. The second arrow is given by the bracket $g \mapsto g\circ \phi - \phi\circ g$ which parametrises deformations of the Higgs field $\phi:\pi_*\E_\phi \rightarrow \pi_*\E_\phi\otimes K_S$.  We remark that this diagram equals its own Serre dual (i.e. replacing all objects by their duals gives the same triangle, just shifted) and refer to [\cite{TT}, pp.16-19] for a proof of these statements.

\section{The involution}
\setcounter{subsection}{-1}
\subsection{Summary}
We will define an involution $\iota:\N\rightarrow \N$ extending the classical map $$(E,\phi) \mapsto (E^*,-\phi^*)$$ to all torsion free pairs Higgs pairs $(E,\phi)$ on $S$. This will be phrased as $$\iota: \E_\phi \mapsto \sigma_{{\tr \phi}}^*\E_\phi \otimes \pi^*\det(\pi_*\E_\phi)^{-1}$$  for their corresponding spectral sheaves on $X$, where $\sigma_{{\tr \phi}}$ translates the points in the fibres of $\pi:X \rightarrow S$ by $\tr\phi$. We will define this as a functor of families of spectral sheaves, such that $\iota$ above is its classifying map.
\subsection{The involution for vector bundles}
\begin{dfn}
    Assuming $\rk(E),\deg(E)$ to be coprime, let $(\mathsf{E},\Phi)$ be a family parametrising Higgs bundles $(E,\phi)$ on $S$. As each $E$ is locally free, so is $\mathsf{E}$, thus $\mathsf{E^{*,*}}\cong \mathsf{E}$ holds and $$(\mathsf{E},\Phi) \mapsto (\mathsf{E}^*,-\Phi^*)$$
	defines an involution. 
\end{dfn}
\begin{rmk}
	Assuming $\N$ consists of pairs with all $E$ locally free (or shrink it such that it does) and that $(E^*,-\phi^*)$ is also stable, then fixing the invariants $(2,c_1,c_2)$ for $(E,\phi)$ as in Sec. \ref{chern classes}, this defines a classifying map $$\N(2,c_1,c_2) \rightarrow \N(2,-c_1,c_2)$$
	with square equal to the identity, which is on closed points given by $$[(E,\phi)] \mapsto [(E^*,-\phi^*)].$$
    As in this case Gieseker = slope stable, it is enough to check that this map preserves slope stability: Ineed, let $(E,\phi)$ be stable and let $E^* \rightarrow Q$ be a Higgs quotient, $Q^*\subset E$ and $\mu(Q)=-\mu(Q^*)>-\mu(E)=\mu(E^*)$, so $(E,-\phi^*)$ is indeed stable.
\end{rmk}
We will modify this involution in the next part and furhermore, as we are interested in the locus of $\N$ where $\det(E)\cong \OO_S$, it is sufficient to restrict entirely to $$c_1(E)=0$$ for the rest of this discussion. In this case, the classifying map is a genuine involution on $\N(2,0,c_2)$. 
In order to generalise to torsion free $E$ and hence to all of $\N$, we need some linear algebra:
\subsection{Skew maps }\label{locallyfree}
The reader only interested in the results of this rather technical subsection might jump directly to the final Rmk. \ref{rmkgeneraliota}.\\ 
Let $\alpha:E\rightarrow E^*$ be a map from a locally free rank $2$ sheaf $E$ to its dual. Then $\alpha \mapsto \alpha^*$ defines an involution on $\Homm(E,E^*)$, i.e. splits $$\Homm(E,E^*)\cong \mathcal{S}ym(E^*)\oplus \wedge^2 E^*$$ into $\pm1 $ eigenspaces: sections of the former are self-dual maps $\alpha=\alpha^*$ and sections of the latter skew maps $\alpha^*=-\alpha$. \\
As $E$ is rank $2$, a section $\alpha$ of $ \wedge^2E^*$  defines an isomorphism whenever it is non-zero, which gives a canonical identification $$E \otimes( \wedge^2E^*) \xrightarrow{\sim} E^*$$  by evaluation. \\
We explain sections of $ \wedge^2 E^*$ in more detail: Fix an open $U\subset S$. A section $\alpha=\alpha_1\wedge \alpha_2$ of $ \wedge^2 E^*(U) $ becomes a skew map $E(U) \rightarrow E^*(U)$ via $$[\alpha_1\wedge \alpha_2] (e) :=\alpha_1(e)\alpha_2-\alpha_2(e)\alpha_1\in E^*(U)$$ for $e \in E(U)$.\\
We need the following fact relating trace and wedge product and refer the reader to \cite[pp.111-112]{FH} for a proof:
\begin{lem}\label{Fulton}
	Let $E$ be a vector bundle of rank $r$ and $\phi: E \rightarrow E\otimes K_S$. Then $\sum_i e_1\wedge \dots \wedge \phi(e_i)\wedge \dots \wedge e_r =\tr(\phi) e_1\wedge \dots \wedge e_r$ for sections $e_i$ of $E$. 
\end{lem}
Now let $\alpha:E(U)\rightarrow E^*(U)$ be skew-map over $U$ corresponding to  a section $\alpha_1\wedge \alpha_2 \in \wedge^2 E^*(U)$. 
\begin{lem}\label{key lemma}
	We have
	$$\alpha\phi-(\alpha\phi)^*=\tr(\phi)\alpha$$
	as elements of $E^*\otimes E^*\otimes K_S(U)$.
\end{lem}
\begin{proof}
	As $\alpha$ is skew, we only need to show $ \alpha\phi- (-\phi^*\alpha)=\tr(\phi)\alpha $.\\
	Indeed, for $e \in E(U)$ and as $\alpha^*=-\alpha$, the LHS equals
	%then $(\alpha\phi) (a) $ is given by $([\alpha_1\wedge \alpha_2]\circ \phi)(a)=\alpha_1(\phi(a))\alpha_2-\alpha_2(\phi(a))\alpha_1 \in V^*$. Similarly, $-\phi^* \alpha_1(a)= -\phi^* [\alpha_1\wedge \alpha_2](a)=-\alpha_1(a)\phi^*\alpha_2 +\alpha_2(a)\phi^*\alpha_1$. Thus their difference is 
	\begin{align*}
	(\alpha\phi +\phi^*\alpha)(e)&= (( \phi^*\alpha_1)(e)\alpha_2-\alpha_2(e)\phi^*\alpha_1)+(\alpha_1(e)\phi^*\alpha_2-(\phi^*\alpha_2)(e)\alpha_1)\\&= [\phi^*\alpha_1\wedge \alpha_2 + \alpha_1\wedge \phi^*\alpha_2](e) \\&= [\tr(\phi^*)\alpha_1\wedge \alpha_2](e)= \tr(\phi)[\alpha_1\wedge \alpha_2](e) =\tr(\phi)\alpha(e)
	\end{align*}
	where we used Lem. \ref{Fulton} in the equality between second and third line.
\end{proof} 
Lastly, we find:
\begin{cor}{\label{torsionfree}}
	The following diagram 
	$$
	\begin{tikzcd}[column sep=14ex]
	E \otimes \bigwedge ^2E^*\arrow[r,"(\phi-\tr(\phi)\cdot \id)\otimes1 "] \arrow[d, "\cong"]  & E \otimes K_S\otimes  \bigwedge^2E^*  \arrow[d,"\cong"] \\ 
	E^* \arrow[r,"-\phi^*"] & E^*\otimes K_S \\ 
	\end{tikzcd}
	$$
	commutes.
\end{cor}
\begin{proof}
	Indeed, for $e\otimes \alpha$, going down the LHS side gives $-\phi^*\alpha(e)$. On the RHS, we get $\alpha(\phi-\tr(\phi)\cdot \id)(e)=\alpha\phi(e)-\tr(\phi)\alpha(e)=-\phi^*\alpha(e)$ by Lem. \ref{key lemma}. 
\end{proof}
\begin{rmk}\label{rmkgeneraliota}
	This implies that $$(E^*,-\phi^*) \cong (E\otimes \wedge^2E^*,(\phi-\tr(\phi)\cdot \id)\otimes 1)$$ are isomorphic Higgs bundles. Splitting $$\mathcal{E}nd(E)\otimes K_S=(\mathcal{E}nd_0(E)\otimes K_S)\oplus K_S\cdot \id,$$ we can write $\phi=\phi_0 \oplus \frac{1}{2}\tr(\phi)\cdot \id$ and see that $ \phi -\tr(\phi)\cdot \id=\phi_0\oplus- \frac{1}{2}\tr(\phi)\cdot \id$.
\end{rmk}
	\begin{rmk}
		Using this, we may redefine the involution
		$$(E,\phi) \mapsto (E^*,-\phi^*)$$ on Higgs bundles (again because locally frees are reflexive) as 
		\begin{equation*}
		(E,\phi) \mapsto (E\otimes \wedge^2E^*,(\phi-\tr(\phi)\cdot \id) \otimes 1)
		\end{equation*}
		under the isomorphism stated in Cor. \ref{torsionfree}. This observation allows us now to generalise $\iota$ to torsion free sheaves:
	\end{rmk}
	\subsection{Torsion free sheaves}
	\begin{dfn}
		A coherent sheaf $E$ on $S$ is torsion free, if all of its torsion sections $e$ are zero: For an open $U\subset S$, a section $e$ of $E(U)$ is called \textit{torsion} if there is some $s\in U$ such that the image of $e$ in the local $ \OO_s$-module $E_s$ has torsion. Equivalently, $e$ is torsion if the image of $e$ in $E_\eta$ is zero, where $\eta$ denotes the generic point of $S$.\footnote{We remark that $\eta$ exists, as $S$ is integral and therefore admits a generic point.} \\
		We define $\rk(E)$ generically, i.e. it is the rank of the $\OO_\eta$-module $E_\eta$.
	\end{dfn}
	As $E$ is not necessarily locally free, we need the notion of homological dimension:
	\begin{dfn}
		For a point $s\in S$ and any $\OO_s$-module $M$, we define its homological dimension $\mathrm{dh}(M)$ as the minimal length of a projective resolution of $M$. For a coherent sheaf $E$ on $S$, we define the homological dimension to be
		 $$\mathrm{dh}(E):=\mathrm{max}\{\mathrm{dh}(E_s): s \in S \}.$$
	\end{dfn}

Now let $E$ be torsion free, then we have \footnote{See \cite[pp.4-6]{HL} for the a precise introduction to resolutions of sheaves and their homological dimension.}
	$$\mathrm{dh}(E) \leq \dim(S) -1.$$
	As $S$ is a surface, we have $\mathrm{dh}(E) \leq 1$ and as $S$ is smooth, the projective resolution can be chosen to consist of locally frees. \\
	Thus either $E$ itself is locally free or $\mathrm{dh}(E)=1$, i.e. there is a $2$-step resolution $E^{-1}\hookrightarrow E^0 \twoheadrightarrow E$ of locally frees $E^{i}$.
	\begin{dfn}\label{determinantbundle}
		We define the determinant bundle for a torsion free $E$ as $$\det(E):=\det(E^0)\otimes \det(E^{-1})^{-1},$$
		which is independent of the choice of resolution and agrees if $\rk(E)=2$ with $\wedge^2(E)$ whenever $E$ is locally free, see e.g. \cite[pp.149-152]{K}. 
	\end{dfn}
	\subsection{The generalised involution}
	This allows us a generalisation of $\iota$: 
	\begin{dfn}\label{geninvol}
		Denoting $\phi^{\mathfrak{t}}:=\phi-\tr(\phi)\cdot \id$, we can extend the involution $$\iota:(E,\phi) \mapsto (E^*,-\phi^*)$$ of \textit{Higgs bundles} to all torsion free \textit{Higgs sheaves} on $S$ by the formula 
		$$\iota: (E,\phi) \mapsto (E\otimes \det(E)^{-1},\phi^{\mathfrak{t}} \otimes 1),$$
		which now makes sense for all rank $2$ torsion frees and extends the original involution defined for Higgs bundles by the diagram stated in Cor. \ref{torsionfree}. We also see that the numerical invariants $(2,0,c_2)$ of $E$ are again fixed under $\iota$. 
	\end{dfn}
	\begin{rmk}\label{squareid}
		We remark that we actually have $\iota^2=\id$. Indeed, we note $\phi^{\mathfrak{t},\mathfrak{t}}=\phi$ and compute 
		\begin{align*}
		\iota^2(E,\phi)&=\iota (E\otimes \det(E)^{-1},\phi^\mathfrak{t}\otimes 1)\\
		&=(E\otimes \det(E)^{-1}\otimes \det(E \otimes \det(E)^{-1})^{-1},(\phi^\mathfrak{t}+\tr(\phi)\cdot \id)\otimes 1)\\&=(E,\phi).
		\end{align*}
	
	\end{rmk}
\subsection{The involution for spectral sheaves}
To define $\iota$ for their corresponding spectral sheaves $\E_\phi$ on $X$, we define the generalised involution directly in families $\EE$ and start with some definitions: 
\begin{dfn}\label{sigma}
	We may write a closed point $x\in X$ as $x=(s,t)$ with $s=\pi(x) \in S$, $t\in K_S|_s$.\footnote{In the complex analytic or étale topology. Alternatively, $X\rightarrow S$ is locally the cone $\mathbf{Spec}(A[z])$ over $\mathbf{Spec}(A)$ and $\sigma_{{\tr \phi}}$ is then given by the $A$-algebra homomorphism $z \mapsto z +\tr(\phi)$ which glues to all of $X$ because $\phi$ does.} For a fixed Higgs field $\phi$, we  define the \textit{trace shift} by $\phi$
	$$\sigma_{\tr\phi}: X \rightarrow X,$$ 
	as $$(s,t) \mapsto (s,t-\tr(\phi_s)).$$
We see that $\sigma_{\tr\phi}$ preserves the fibres of $X \xrightarrow{\pi} S$, i.e. $\pi \sigma_{\tr\phi}=\pi$. It is an invertible map on $X$ with inverse $\sigma_{-\tr\phi}$, acting on spectral sheaves $\E_\phi$ as $$ \E_\phi \mapsto \sigma_{\tr\phi}^* \E_\phi.$$ 
\end{dfn}
\begin{lem}\label{traceshiftspectral}
	We have $\sigma_{\tr\phi}^*\E_\phi=\E_{\phi^\mathfrak{t}}$, so $\sigma_{\tr\phi}^*\E_\phi$ is the spectral sheaf for the Higgs pair $(E,\phi^\mathfrak{t})$. 
\end{lem}
\begin{proof}
	Let $\tau$ be the tautological section of $\pi^*K_S$ on $X$ and choose local coordinates $(s,t)$.
	Recall $\tau_{(s,t)}=t$ and $\pi_*(\tau\cdot \id)_s=\phi_s$. So $(\sigma_{\tr\phi}^*\tau)_{(s,t)}=\tau_{(s,t-\tr(\phi_s))}=t-\tr(\phi_s)$. Thus $\sigma_{\tr\phi}^*(\tau\cdot\id)$ acting on $\sigma_{\tr\phi}^*\E_\phi$ gives $$\pi_*(\sigma_{\tr\phi}^*(\tau \cdot \id))=\phi-\tr(\phi)\cdot \id=\phi^{\mathfrak{t}}.$$ 
	In addition we compute $\pi_*(\sigma_{\tr\phi}^*\E_\phi)=\pi_*\sigma_{\tr\phi,*}(\sigma_{\tr\phi}^*\E_\phi)=\pi_*(\sigma_{\tr\phi,*}\sigma_{\tr\phi}^*)\E_\phi=\pi_*\E_\phi=E $, which shows that $\sigma_{\tr\phi}^*\E_\phi$ is the spectral sheaf for the Higgs pair $(E,\phi^\mathfrak{t})$. 
\end{proof} \noindent
\begin{rmk}
$\sigma_{{\tr \phi}}$ acting on spectral sheaves $\E_\phi$ defines an involution: Indeed, we compute $$\sigma_{{\tr \phi}}^*(\sigma_{{\tr \phi}}^{*}\E_\phi)=\sigma_{{\tr \phi}}^*\E_{\phi^\mathfrak{t}}=\E_{\phi^\mathfrak{t,t}}=\E_\phi.$$
\end{rmk}

\begin{dfn}
	On $X\times \N$, we define $$(\sigma_{{\tr \phi}}\times \id): X\times \N \rightarrow X\times \N,$$ 
	$$ (x,[\E_\phi]) \mapsto (\sigma_{{\tr \phi}}(x),[\E_\phi]).$$
\end{dfn}
Choosing a (twisted) universal family $\EE$, we consider the functor $$ \EE \mapsto (\sigma_{{\tr \phi}}\times \id)^*\EE$$ and claim:
\begin{claim}
	 This functor admits a classifying map $\sigma: \N \rightarrow \N$ on closed points acting as $$[\E_\phi] \mapsto [\sigma_{{\tr \phi}}^*\E_\phi].$$
\end{claim}
\begin{proof}
	We note that $\EE \mapsto (\sigma_{{\tr \phi}}\times \id)^*\EE$ preserves flatness and $\sigma_{{\tr \phi}}$ fixes the Chern classes defined in Sec. \ref{chern classes}. Thus, the functor defines a classifying map $\sigma: \N \rightarrow \N$, such that $(\id\times \sigma)^*\EE\cong (\sigma_{{\tr \phi}}\times \id)^*\EE\otimes \mathcal{L}$ for some line bundle $\mathcal{L}$ pulled back from $\N$. Choosing an open $U$ where $\mathcal{L}|_U\cong \OO_U$ holds, allows us to read off the coordinates of $\sigma$ as $[\E_\phi]\mapsto [\sigma_{\tr \phi}^*\E_\phi].$
\end{proof}
\begin{dfn}\label{redefinesigma}
	We define the lift of $\sigma$ to $X\times \N$ as the map $$\sigma:=\sigma_{-\tr\phi}\times \sigma:X\times \N \rightarrow X \times \N$$ and claim:
\end{dfn}
\begin{lem}\label{sigmaequivariance}
	There exists an open $U$ on $X\times \N$ and an equivariant structure for $\EE|_U \mapsto \sigma^*\EE|_U$, lifting $\EE$ to $\mathbf{D}^b(U)^{\langle \sigma \rangle}$.
\end{lem}
\begin{proof}
	The previous claim gave us an isomorphism
	\begin{align}\label{twistedlinear}
	(\id \times \sigma)^*\EE\cong(\sigma_{\tr\phi}\times \id)^*\EE\otimes \mathcal{L}
	\end{align}
	for a line bundle $\mathcal{L}$. Applying $(\sigma_{-\tr\phi}\times \id)^*$ on both sides yields over $U$ where $\mathcal{L}|_U\cong \OO_U$ the isomorphism $$(\sigma_{-\tr\phi}\times\sigma)^*\EE|_U\cong \EE|_U,$$
	where the LHS is $\sigma^*\EE$ restricted to $U$.
	 Thus, we can pick an isomorphism $\Psi_U:\EE|_U \cong \sigma^*\EE|_U$. By construction, its inverse is $\sigma^*\Psi_U$, so there is a commutative square
	\begin{equation} 
	\begin{tikzcd}
	\EE|_U \arrow[r,"\Psi_U"] \arrow[dr,equal]& \sigma^*\EE|_U \arrow[d,"\sigma^*\Psi_U"]\\
	&(\sigma^*)^2\EE|_U
	\end{tikzcd}
	\end{equation}
	proving the claim in the sense of Def. \ref{equi}.
\end{proof}
\begin{rmk}
	In fact, the proof constructs a twisted isomorphism $\Psi: \EE \cong \sigma^*\EE\otimes \mathcal{L}$. 
\end{rmk}
\begin{cor}\label{adjoint}
	There is a lift of $\mathbf{R}\mathcal{H}om_{p_X}(\EE,\EE)$ to $\mathbf{D}^b(\N)^{\langle \sigma \rangle}$. 
\end{cor}

\begin{proof}
	We note $\sigma^2=\id$. Then the adjoint action $g \mapsto \Psi g\Psi^{-1}$ composed with the isomorphism induced by the evaluation $\mathcal{L}\otimes\mathcal{L}^{-1}\cong \OO$ defines a linearisation $$\sigma_*: \mathbf{R}\mathcal{H}om(\EE,\EE) \xrightarrow{\cong
	} \mathbf{R}\mathcal{H}om(\sigma^*\EE\otimes \mathcal{L},\sigma^*\EE\otimes \mathcal{L})\xrightarrow{\cong} \mathbf{R}\mathcal{H}om(\sigma^*\EE,\sigma^*\EE)$$ with square equal to the identity. Then use functorialty of the push-dowm $\mathbf{R}p_{X,*}$ to $\N$. 
\end{proof}
\begin{rmk}
	Although $\EE$ may not be defined everywhere, the sheaves $\Extt_{p_X}^{i}(\EE,\EE)$ are and this construction glues to an action over all of $X\times \N$, as stated in \cite[Sec. 10.2]{HL}.
\end{rmk}

	\begin{rmk}
		The map $\sigma$ shifts $\mathrm{supp}(\E_\phi)\subset X$ according to the action $\phi \mapsto \phi^\mathfrak{t}$ on Higgs fields. In order to express the line bundle twist $E\mapsto E\otimes \det(E)^{-1}$ in terms of spectral sheaves we define:
	\end{rmk}
\begin{dfn}
	For a universal family $\EE$ of spectral sheaves, set $\LLL:=  \pi^*\det(\pi_*\EE)^{-1}$
	and define the functor $$\EE \mapsto \EE\otimes \LLL.$$ 
	Similarly to the case before, this functor preserves flatness and defines a classifying map
	\begin{align*}
	\lambda: \N &\rightarrow \N,\\
	[\E_\phi] &\mapsto [\E_\phi\otimes\pi^*\det(\pi_*\E_\phi)^{-1}].
	\end{align*}
\end{dfn}	
	\begin{rmk}
		Under the restriction $c_1(E)=0$ we immediately observe that the $c_i(\E)$ as given in \ref{chern classes} are fixed under $\lambda$, so the map is a genuine involution on $\N$, preserving Gieseker stability.   
	\end{rmk}
	As in Rmk. \ref{squareid}, we observe that $\lambda^2=\id$. Again, by base change along $p_X$ we get an induced endomorphism $(\id\times \lambda)$ on $X\times\N$ and will use the letter $\lambda$ again for both maps.

We observe that $\sigma,\lambda$ commute with each other:
\begin{lem}\label{commutesigmalambda}
	We have $\sigma\lambda=\lambda\sigma$ on $\N$ and $(\sigma \lambda)^2=\id$.
\end{lem}
\begin{proof}
	We choose a representative of a class $[\E_\phi]\in \N$ and compute\\
	$$\sigma\lambda(\E_{\phi}) =\sigma_{{\tr \phi}}^*\E_\phi\otimes \sigma_{{\tr \phi}}^*\pi^*\det(\pi_*\E_\phi)^{-1}=\E_{\phi^\mathfrak{t}}\otimes \pi^*\det(\pi_*\E_{\phi^\mathfrak{t}})^{-1}=\lambda\sigma(\E_\phi),$$
	as $\sigma_{{\tr \phi}}^*\pi^*=(\pi \circ \sigma_{{\tr \phi}})^*=\pi^*$ and $\pi_*\E_{\phi^\mathfrak{t}}=\pi_*\E_\phi$ following from Lem. \ref{traceshiftspectral} above. \\
	This composition has square equal to the identity as $(\sigma\circ \lambda)^2=\sigma\circ \lambda\circ  \sigma\circ\lambda=\sigma^2\circ \lambda^2=\id$. 
\end{proof}
This allows us to define the generalised involution for spectral families by composing the two functors:
\begin{dfn}\label{spectralfamilies}
	For a universal family $\EE$ of spectral sheaves $\E_\phi$ on $X$ parametrised by $\N$ we define 
	\begin{align*}
		\iota: \EE \mapsto (\sigma_{{\tr \phi}}\times \id)^*\EE\otimes \LLL
	\end{align*}
\end{dfn}

	We end this section by showing that the functor $\iota$ induces a classifying map $\iota: \N \rightarrow \N$ agreeing on closed points with the involution of Def. \ref{geninvol}, 
		$$\iota: (E,\phi) \mapsto (E\otimes \det(E)^{-1},\phi^{\mathfrak{t}} \otimes 1),$$
	which generalised the original map $$(E,\phi) \mapsto (E^*,-\phi^*)$$
	to torsion free Higgs pairs:
\begin{thm} \label{spectralmoduli}
	The functor $$\iota: \EE \mapsto (\sigma_{\tr\phi}\times \id)^*\EE\otimes \LLL$$ for families $\EE$ of spectral sheaves parametrized by $\N$ defines a classifying map $\iota: \N \rightarrow \N$, such that $ [\E_\phi] \mapsto \iota [\E_\phi]$ gives under the spectral correspondence the map $$\iota: (E,\phi) \mapsto (E\otimes \det(E)^{-1},\phi^{\mathfrak{t}} \otimes 1),$$
	 extending $$(E,\phi) \mapsto (E^*,-\phi^*)$$ to torsion free Higgs pairs, as defined in Def. \ref{geninvol}. More precisely, there is a commutative square
	\begin{equation}\label{diagraminvol}
	\begin{tikzcd}
	X\times \mathcal{N} \arrow[r,"\id\times \iota"]\arrow[d,"\pi \times \id"] & X\times \mathcal{N} \arrow[d,"\pi \times \id"]\\
	S\times \mathcal{N} \arrow[,r,"\id\times\iota"] &S\times \mathcal{N}. \\
	\end{tikzcd}
	\end{equation}
	where the upper $\iota$ is the action on spectral sheaves and the lower $\iota$ the one on Higgs pairs. 
\end{thm}
	\begin{proof}
		We have already discussed the classifying maps $\sigma, \lambda: \N \rightarrow \N$ arising from the functors $\EE\mapsto (\sigma_{{\tr \phi}}\times \id)^*\EE$ and $\EE\mapsto \EE\otimes \LLL$ respectively. Thus, the functor $\iota$ admits their composition $\iota: \sigma\circ \lambda: \N \rightarrow \N$ as classifying map satisfying $\iota^2=\id$. \\
		We need to show that $\iota$ induces under the spectral correspondence $[\E_\phi]=[(E,\phi)]$ the generalised involution $\iota$ of torsion free Higgs pairs $(E,\phi)$, i.e. we need to show that $$\iota(\E_\phi)=\E_{\phi^\mathfrak{t}}\otimes \pi^*\det(\pi_*\E_\phi)^{-1}$$ corresponds to the involuted Higgs pair $$\iota(E,\phi)=(E\otimes \det(E)^{-1},\phi^\mathfrak{t}\otimes 1).$$
		We have already seen that $\pi_*\E_{\phi^\mathfrak{t}}=\pi_*\E_\phi=E$ in Lem. \ref{traceshiftspectral}.\\
		Then $\pi_*(\iota\E_\phi)=\pi_*(\E_{\phi^\mathfrak{t}} \otimes \pi^*\det(\pi_*\E_{\phi^	\mathfrak{t}})^{-1})=E\otimes \det(E)^{-1}$. 
		Now the action of the tautological endomorphism $\tau \cdot \id $ on  $\E_\phi$  induces $\tau \cdot \id \otimes 1$ acting on $\E_\phi\otimes \pi^*\det(\pi_*\E_\phi)^{-1}$,  so $\pi_*\sigma_{\tr\phi}^*(\tau\cdot \id\otimes 1)= \phi^\mathfrak{t}\otimes 1$. Thus $\E_{\phi^\mathfrak{t}} \otimes \pi^*\det(\pi_*\E_\phi)^{-1}$ on $X$ corresponds to the Higgs pair $(E\otimes\det(E)^{-1},\phi^\mathfrak{t}\otimes 1)$ on $S$, showing commutativity of Diag. \ref{diagraminvol}.\\
		Both maps have square equal to the identity and we know that $\iota$ acting on Higgs pairs generalises $(E,\phi) \mapsto (E^*,-\phi^*)$ from locally free pairs to torsion frees, as shown in Cor. \ref{torsionfree}. This finishes the theorem.
	\end{proof}

%\subsection{Action on the moduli}
%After having found the right definition of $\iota$, we discuss how this map acts on $\N$. Namely, let $\N(2,c_1,c_2)$ be the moduli space of torsion frees $(E,\phi)$ on $S$. We see that $\iota$ defines a map $$\iota: \N(2,c_1,c_2) \rightarrow \N(2,-c_1,c_2)$$ as we compute 
%\begin{align*}
%&c_1(E\otimes \det(E)^{-1})=c_1(E)-2c_1(\det(E))=c_1(E)-2c_1(E)=-c_1(E) \; \text{and}\\
%&c_2(E\otimes \det(E))^{-1})=c_1(\det (E))^2-c_1(\det(E)) c_1(E)+c_2(E)=c_2(E)
%\end{align*}
%From Sec. \ref{chern classes}, we get a similar involution on the Chern classes of $\E$. 	

\section{Deformations and the fixed locus}\label{Artinian}
\setcounter{subsection}{-1}
\subsection{Summary}
To simplify notation, we omit the subscript $\phi$ from $\E_\phi$ whenever possible.\\ 
We will see how $\iota$ acts on first order deformations $\Ext^1(\E,\E)$ of $\E$ and will identify one of the components of the scheme-theoretic fixed locus $\N^\iota$ as $$\N^\perp=\{(E,\phi): \det(E)\cong \OO_{S}, \tr(\phi)=0\}.$$
\begin{rmk}
Although this is merely the set theoretical description of $\N^\perp$, we will see that its schematic structure is inherited by the fixed point set $\N^\iota$, being an open and closed therein. Further note that we do not make any assumptions on the  isomorphism $\det(E)\cong\OO_S$ .
\end{rmk}
\begin{rmk}
In terms of spectral sheaves, this is equivalent to define $\N^\perp$ as those $\E$ that have "centre of mass zero" on each fibre of $X\rightarrow S $ and $\det(\pi_*\E)\cong \OO_S$. By the center of mass of $\E$ we mean the sum of the points of $\mathrm{supp}(\E)$ on each fibre $X\rightarrow S$, weighted by multiplicity.
\end{rmk}
\subsection{Artinian families}
Let $A$ be an Artinian local ring and denote by $X_A:=X\times \mathbf{Spec}(A)$ the product. In Sec. \ref{flatfamilies} we have recalled the notion of Artinian families $\E_A$ of spectral sheaves $\E$ on $X$ for a local Artinian $\mathbf{C}$-algebra $A$. As the universal family $\EE$ is locally well-defined on some open containing $\mathbf{Spec}(A)$, an Artinian family $\E_A$ over $X_A$ is equivalent to a morphism $f: \mathbf{Spec}(A)\rightarrow \N$ such that $(\id\times f)^*\EE\cong \E_A$.\\
Let $X_A \xrightarrow{\pi_\A} S_A$ be the base change of $\pi \times \id_{\mathcal{N}} $ along $\id_S \times f$.\\
An Artinian family over $X_A$ is the same as a family of Higgs pairs $({E}_A,\phi_A)$ over $S_A$. $\phi_A$ defines an invertible map $\sigma_{\tr \phi_\A}: X_A \rightarrow X_A$ shifting by $-\tr(\phi_A)$ on the fibres of $\pi_A: X_A \rightarrow S_A$ and we define the line bundle $L_A:=\det(E_A)^{-1}$ on $S_A$.\\
We identify a single spectral sheaf $\E$ on $X$ with its push-forward $i_{A,*}\E$ on $X_A$ for $i_A: X= X\times 0 \hookrightarrow X_A$ the inclusion of the closed point $0$ of $\mathbf{Spec}(A)$, analogously for $\iota\E$.  Note that by base change along $i_A$, we have $\sigma_{\tr\phi_A}i_A=i_A\sigma_{\tr\phi}$ and $\pi_A i_A=i_A \pi$ as displayed in the following diagram:
\begin{center}
	\begin{tikzcd}
	X \arrow[r,"i_A",hook] \arrow[d,"\sigma_{{\tr \phi}}",] & X_A \arrow[d,"\sigma_{{\tr\phi}_A}"] & &X \arrow[r,"i_A", hook] \arrow[d,"\pi"]& X_A \arrow[d,"\pi_A"]\\
	X \arrow[r,"i_A",hook] & X_A & &S \arrow[r,"i_A",hook] &S_A
	\end{tikzcd}
\end{center}

\begin{rmk}
	We see that $\iota$ acts on Artinian families as
	$$ {\E_A} \mapsto \sigma_{{\tr\phi}_A}^*{\E_A} \otimes \pi_A^*{L_A}.$$ 
\end{rmk}
\begin{prop}
	For $A=\mathbf{C}[t]/(t^2)$ this is the differential of $\iota$, $$(d\iota)_{[\E]}: {\Ext}^1(\E,\E)\rightarrow {\Ext}^1(\iota\E,\iota\E)$$
\end{prop}
\begin{proof}
	Take an exact sequence $0\rightarrow \E \rightarrow {\E_A}\rightarrow \E \rightarrow 0$ on $X_A$ which stays exact after applying $\sigma_{{\tr\phi}_A}^*$ and  $\_\otimes \pi^*{L_A}$.  Furthermore, we compute for $\E=i_{A,*}\E$
	\begin{align*}
	&\sigma_{\tr\phi_{A}}^*i_{A,*}\E\otimes \pi_A^*L_A=i_{A,*}\sigma_{\tr \phi}^*\E\otimes \pi_A^*L_A\\
	&=i_{A,*}(\sigma_{\tr\phi}^*\E\otimes i_{A}^*\pi_A^*L_A)=i_{A,*}(\sigma_{\tr\phi}^*\E \otimes \pi^*L)=i_{A,*}(\iota\E).
	\end{align*}
	
	This gives the differential $(d\iota)_{[\E]}: \Ext^1(\E,\E)\rightarrow \Ext^1(\iota\E,\iota\E)$ acting on short exact sequences as 
	\begin{center}
		\begin{tikzcd}
		0 \rightarrow \E \rightarrow \E_A \rightarrow \E \rightarrow 0 \arrow[d,mapsto,"(d\iota)_{[\E]}"]\\
		0 \rightarrow \iota\E \rightarrow  \sigma_{{\tr \phi}_A}^*{\E_A}\otimes \pi_A^*{L_A} \rightarrow \iota \E \rightarrow 0.
		\end{tikzcd}
	\end{center}
\end{proof}

\subsection{$\mathbf{SU}(2)$-Higgs pairs}
The following part will be phrased in terms of Higgs sheaves, where we investigate the scheme-theoretic fixed locus $\N^\iota$ and its relation to $\mathbf{SU}(2)$-Higgs pairs, which are defined as $$\N^\perp:=\{(E,\phi): \det(E)\cong \OO_S \; \text{and} \; \tr(\phi)=0\} \subset \N.$$
Let $(E,\phi)$ a representative of $[(E,\phi)] \in \N^\iota$: Then $E\cong E\otimes \det(E)^{-2}$, hence taking determinants gives  $\det(E)\cong \det(E \otimes \det(E)^{-1})= \det(E)\otimes \det(E)^{-2}$, so we see that $\det(E)$ is a $2$-torsion line bundle. We also find $\phi = \phi^\mathfrak{t}$, so $\tr(\phi)=0$. If in addition $\det(E)\cong \OO_{S}$ holds, we conclude that $[(E,\phi)] \in \N^\perp$.We can state the following:
\begin{prop} \label{fixed locus}
	The $\iota$-fixed locus $\mathcal{N}^\iota$ consists of trace-free Higgs pairs $(E,\phi)$ where $\det(E)$ is $2$-torsion. Conversely, if $[(E,\phi)] \in \N^\perp$, then there exists an isomorphism $\iota (E,\phi)\cong (E,\phi)$, so $[(E,\phi)] \in \N^\iota$.
\end{prop}
\begin{proof}
	We have already described $\N^\iota$ and divide the proof showing that $ \N^\perp \subset \N^\iota$ into two steps:\\
	We start with locally free sheaves, as this case is more illusive: we show first that the class of a pair $(E,\phi)$ with $\det(E)\cong \OO_S$ and $\tr(\phi)=0$ is contained in $\N^\iota$. Recall that for a locally free $E$, $\iota$ can be written as $(E,\phi) \mapsto (E^*,-\phi^*)$.  \\
	In the second step, we will generalise to torsion frees and show that an Artinian family $(E_A,\phi_A)$ satisfying $\tr(\phi_A)=0$ and $\det(E_A)\cong \OO$ is contained in $\N^\iota$. We use Artinian families here as this implies that $\N^\perp \subset \N^\iota$ is open.\\
	Now let $[(E,\phi)] \in \N^\perp$ with $E$ locally free. Now the second summand in $\Homm(E,E^*)=\mathcal{S}ym^2(E^*) \oplus \det(E)^{-1}$ is trivial, thus it admits a nowhere vanishing section $\alpha$, that is, a skew isomorphism $\alpha: E \xrightarrow{\sim} E^*$ endowing $E$ with a symplectic structure. \\
	As $\tr(\phi)=0$, Lem. \ref{key lemma} reads as $\alpha \phi = (-\phi^*)\alpha$, so $$\alpha: (E,\phi) \xrightarrow{\sim} (E^*,-\phi^*)$$ defines an isomorphism of Higgs bundles.\footnote{We remark that over $\N^\perp$, the map $\alpha: (E,\phi) \xrightarrow{\sim} (E^*,-\phi^*) $ linearises the functor $(E,\phi) \mapsto \iota(E,\phi)$.}\\
	Now let $(E_A,\phi_A)$ be an Artinian family of torsion frees corresponding to a map $\mathbf{Spec}(A)\rightarrow \N^\perp$. We need to show this family is $\iota$-fixed:
	As $\tr(\phi_A)=0$, we have $\phi_A=\phi_A^\mathfrak{t}$. Furthermore, there exists a trivialisation $\alpha: \det(E_A)^{-1} \xrightarrow{\sim} \mathcal{O}_{S_A}$ and we claim that
	$$
	\begin{tikzcd}
	{E_A}\otimes \det(E_A)^{-1} \arrow[r,"\id \otimes \alpha"] \arrow[d,"\phi_A^\mathfrak{t}"] & {E_A}  \arrow[d,"{\phi_A}"]& \\
	{E_A}\otimes \det(E_A)^{-1} \otimes K_{S_A}\arrow[r,"\id \otimes \alpha\otimes 1"] & {E_A} \otimes K_{S_A}
	\end{tikzcd}
	$$
	commutes. As $\alpha,\alpha^{-1} \in \mathcal{O}_{S_A}$ and $\phi_A$ is $\mathcal{O}_{S_A}$-linear, we see that  $\alpha\phi_A^\mathfrak{t} \alpha^{-1} =\phi_A^\mathfrak{t} = \phi_A$, so $\id \otimes \alpha$ defines an isomorphism $$(E_A\otimes \det(E_A)^{-1},\phi_A^\mathfrak{t})\xrightarrow{\sim} (E_A,\phi_A)$$ thus $(E_A,\phi_A) \in \mathcal{N}^\iota$. 
\end{proof}
\begin{rmk}
	This shows $\N^\perp$ is open in $\N^\iota$. To identify $\N^\perp$ with a component of $\N^\iota$, we need to show it is also closed:
\end{rmk}
\begin{cor}\label{component}
$\N^\perp \subset \N^\iota$ is also closed.
\end{cor}
\begin{proof}
	Indeed, restricting the map $\det: \N \rightarrow \mathbf{Pic}(S)$ which sends $(E,\phi) \mapsto \det(E)$ to $\N^\iota$ has image in the discrete set of 2-torsion line bundles $\mathbf{Pic}(S)[2]$ and thus decomposes $\N^\iota$ into different components. In particular, $\N^\perp=\det^{-1}([\OO_S])\cap \N^\iota$ is closed. Note that $\N^\iota \subset \N$ itself is closed being the fixed locus of a finite group action.  
\end{proof}
\begin{rmk}
	Equivalently, $\N^\perp \subset \N$ are those spectral sheaves $\E$ that have center of mass equal to zero and $\det(\pi_*\E)\cong \OO_{S}$.
\end{rmk}
\section{The partial Atiyah class}
\setcounter{subsection}{-1}
\subsection{Summary}
After having introduced the Atiyah class in a general setup in Sec. \ref{Atiyahclass}, this section introduces the truncated Atiyah class $\At_{\EE,\N}$ on $\N$. We first deal with the non-projectivity:
\subsection{The projective completion}
As $X$ is the total space of a line bundle, it is non-compact. We embed $j: X\subset \overline{X} :=\mathbf{P}(K_S^*\oplus \OO_S)$ and claim:
\begin{claim}\label{relativediff}
We have $\E\otimes\omega_{\overline{X}}\cong \E$ for single sheaves $\E$ on $X$ and similarly $\EE\otimes \omega_{p_{\overline{X}}} \cong \EE$ for a universal sheaf $\EE$ on $X\times \N$. 
\end{claim}
\begin{proof}
	We may identify the spectral sheaves $\E$ with their pushforward $j_*\E$. Then $\overline{\pi}: \overline{X} \rightarrow S $ is a $\mathbf{P}^1$-bundle containing $X$ as an open. Let $\OO(1)$ be a polarisation\footnote{This could be $\pi^*\OO_S(k)\otimes H^l$ for suitable $k,l>0$ and $H$ the relative hyperplane bundle for $X\xrightarrow{\pi} S$.} on $\overline{X}$ or its pull-back to $\overline{X}\times\N$.\\
	Although $\overline{X}$ is not of Calabi-Yau type, we see that $$\E\otimes \omega_{\overline{X}}\cong j_*(\E\otimes j^*\omega_{\overline{X}})=j_*(\E\otimes \omega_X)=j_*\E= \E,$$ 
	so its canonical $\omega_{\overline{X}}$ is trivial when restricted to $\mathrm{supp}(\E)$. \\
	A universal family $\EE$ on $\overline{X}\times \N$ for spectral sheaves $[\E] \in \N$ is by definition supported on $X\times \N \subset \overline{X}\times \N$. So similarly to the computation above, we observe that 
	$$\EE\otimes \omega_{p_{\overline{X}}} \cong \EE$$ for the \textit{relative} dualising sheaf $\omega_{p_{\overline{X}}}=p_{\overline{X}}^*\omega_{\overline{X}}$ on $\overline{X}\times \N$,
	again because $$(j\times \id)^*p_{\overline{X}}^*\omega_{\overline{X}}\cong {p^*_X}j^*\omega_{\overline{X}}\cong p_{X}^*\omega_X\cong \OO.$$
\end{proof}

%\begin{rmk}
%	The last two facts show that if there is an isomorphism $\Psi: f^*\EE \cong \EE$, then 
%	\begin{equation}\label{functoriality}
%	\begin{tikzcd}
%	 \RHomm(f^*\EE,f^*\EE) \arrow[r,"f_*"] \arrow[d,"f^*\At_{\FF}"]&  \RHomm(\EE,\EE)\arrow[d,"\At_{\EE}"]\\
% f^*\mathbf{L}_{X\times \N}[1] \arrow[r,"f_*"] &  \mathbf{L}_{X\times \N}[1]	
%	\end{tikzcd}
%	\end{equation}
%	commutes, where the upper $f_*$ is the conjugation action induced by $\Psi$. 
%\end{rmk}	
\subsection{The partial Atiyah class}\label{partialAtdef}
We define the partial Atiyah class for $\EE$ using Grothendieck-Verdier duality along $p_X: X\times\N \rightarrow \N$. As $p_X$ is not proper, we embed $X\subset \overline{X}$ again into its projective completion and identify $\EE$ with its push-forward to $\overline{X}\times \N$. We need the following lemma:
\begin{lem}\label{GVD}
As sheaves in $\mathbf{D}^b(\N)$, we have the following identifications
\begin{align*}
    \mathbf{R}p_{\overline{X},*} (\RHomm(\EE,\EE)\otimes \omega_{p_{\overline{X}}}) \cong \mathbf{R}p_{\overline{X},*} (\RHomm(\EE,\EE)) \\\cong \mathbf{R}p_{X,*} (\RHomm(\EE,\EE))
\end{align*}
\end{lem}
\begin{proof}
	The first isomorphism follows directly from Claim \ref{relativediff} as $\EE \otimes \omega_{p_{\overline{X}}} \cong \EE$. For the second identification we use the fact that the derived sheaf $\RHomm(\EE,\EE)$ on $\overline{X}\times \N$ is supported on $X\times \N$. 
\end{proof}
Composing the natural  map $\mathbf{L}_{\overline{X}\times\N} \rightarrow p_{\overline{X}}^*\mathbf{L}_\N$ with $\At_{\EE}$ gives 
$$\mathbf{R}\mathcal{H}om_{p_{\overline{X}}}(\EE,\EE)\rightarrow p_{\overline{X}}^*\mathbf{L}_\N[1].$$
Then Grothendieck-Verdier duality \footnote{For a precise statement of Grothendieck Verdier duality we refer to \cite[pp. 86-90]{H}.} along $p_{\overline{X}}$ gives 
$$\mathbf{R}p_{\overline{X},*} (\RHomm(\EE,\EE)\otimes \omega_{p_{\overline{X}}}) [2] \rightarrow \mathbf{L}_\N.$$
Using Lemma \ref{GVD} above, we call the resulting morphism
$$\At_{\EE, \N}: \mathbf{R}\mathcal{H}om_{p_X}(\EE,\EE)[2] \rightarrow \mathbf{L}_\N$$
the \textit{partial} Atiyah class on $\N$. 
%As this duality is \textit{functorial}, it gives via $p_{X,*}$ the commutative diagram deduced from above remark \ref{functoriality}
%\begin{equation}\label{functorialitypar}
%\begin{tikzcd}
%\RHomm_{p_X}(f^*\EE,f^*\EE)[2] \arrow[r,"f_*"] \arrow[d,"f^*\At_{\EE,\N}"]&  \RHomm_{p_X}(\EE,\EE)[2]\arrow[d,"\At_{\EE,\N}"]\\
%f^*\mathbf{L}_{\N} \arrow[r,"f_*"] &  \mathbf{L}_{\N}	
%\end{tikzcd}
%\end{equation}
\begin{rmk}
	Whenever possible, we omit the subscript $\N$ from notation. 
\end{rmk}		
\begin{prop}\label{partialAtiyah}
	The obstruction theory given by the truncated partial Atiyah class $$\At_{\EE, \N}: \tau^{[-1,0]}\mathbf{R}\mathcal{H}om_{p_X}(\EE,\EE)[2]\rightarrow  \mathbf{L}_\N$$
	admits a $2$-term representation of vector bundles. 
\end{prop}
\begin{proof}
	We identify $\mathbf{R}\Homm(\EE,\EE)$ with its push-forward to $\overline{X}\times \N$ and choose a sufficiently negative finite resolution\footnote{Note that we can assume the finiteness as $\overline{X}\times \N$ is a quasi-projective variety. } $F^{\bullet}$ on $\overline{X}\times\N$ of locally frees representing $\mathbf{R}\mathcal{H}om(\EE,\EE)$ such that 
	\begin{itemize}
		\item the full Atiyah class $\At_{\EE}: \mathbf{R}\mathcal{H}om(\EE,\EE)\rightarrow \mathbf{L}_{\overline{X}\times \N}[1]$ is represented by complexes (see Def.\ref{represbycomplex}) and so is the natural map $p_\N^*\mathbf{L}_{\overline{X}\times \N}\rightarrow \mathbf{L}_\N$. 
		\item the push-downs $p_{\overline{X},*} F^{k}$ to $\N$ are again locally free for all $k$.
	\end{itemize}  
	The latter can be achieved e.g. by making sure $F^{\bullet,\vee}$ has no fibre-wise cohomology $H^{i}(\overline{X}_{n},F^k|_{\overline{X}_n})=0$ for $i>0$ and all $k$. This forces $p_{\overline{X},*} (F^{\bullet,\vee})$ to consist of locally frees, hence the same holds for $p_{\overline{X},*} F^{\bullet}$. Composing $\At_{\EE}$ with $p_{\overline{X}}^*\mathbf{L}_{\overline{X}\times \N}\rightarrow \mathbf{L}_\N$ gives after applying Verdier duality to $p_{\overline{X}}$ the map
	\begin{align}\label{repres}
	\mathbf{R}p_{\overline{X},*}(\mathbf{R}\mathcal{H}om(\EE,\EE)  \otimes \omega_{p_{\overline{X}}})[2] \rightarrow \mathbf{L}_\N
	\end{align}
	By Lem. \ref{GVD}, we have $\mathbf{R}p_{\overline{X},*} (\RHomm(\EE,\EE)\otimes \omega_{p_{\overline{X}}}) \cong \mathbf{R}p_{X,*} (\RHomm(\EE,\EE))$. Thus, truncation to degrees $-1,0$ gives a representation of complexes for the truncated partial Atiyah class on $\N$
	$$\At_{\EE,\N}: \tau^{[-1,0]}\mathbf{R}\mathcal{H}om_{p_X}(\EE,\EE)[2]\rightarrow  \mathbf{L}_\N,$$
	where the LHS is represented by a 2-term complex of vector bundles, as we made sure that $p_{\overline{X},*} F^\bullet[2]$ consist of locally frees. 
	
\end{proof}

\begin{rmk}
	First observe that $\At_{\EE, \N}$ is by \cite[Lem.4.2]{HT} an obstruction theory and now perfect. \\
	We denote by $V^\bullet$ the resulting $2$-term representation of $$\tau^{[-1,0]}\mathbf{R}\mathcal{H}om_{p_X}(\EE,\EE)[2]$$
	and remark that the proof gives a map $$ [V^{-1}\rightarrow V^{0}] \xrightarrow{\psi} [\mathcal{I}/\mathcal{I}^2 \rightarrow \Omega_{\mathcal{A}}|_{\N}]  \hspace{5pt} \in \hspace{5pt} \mathbf{D}^{[-1,0]}(\N).$$
	We will work in the next section with the non-truncated partial Atiyah class only and will often use the dual $$\At_{\EE,\N}: \mathbf{T}_\N \rightarrow \mathbf{R}\mathcal{H}om_{p_X}(\EE,\EE)[1],$$
	denoted by the same symbol.
\end{rmk} 
\begin{dfn} 
	As already introduced in \ref{virtualtangentbundle}, we call $V^\bullet=[V^{-1} \rightarrow V^0]$ the virtual cotangent bundle of $\N$ and its dual $V^{\bullet,\vee}=[V^0 \rightarrow V^1]$ the virtual tangent bundle. 
\end{dfn}
The $2$-term representation of  $\mathbf{R}\mathcal{H}om_{p_X}(\EE,\EE)[1]$ is going to be important for Sec. \ref{finalsec}.

\section{The trace-identity splitting}
\setcounter{subsection}{-1}
\subsection{Summary}
In order to lift $\iota$ to the virtual (co-)tangent bundle, we need to discuss the $\iota$-equivariance for $\mathbf{R}\mathcal{H}om_{p_X}(\EE,\EE)$ and the Atiyah class $\At_{\EE, \N}$.\\
We decompose $\iota=\sigma \circ \lambda$ on $\N$ into line bundle twist $\lambda$ and trace shift $\sigma$ (see Def. \ref{spectralfamilies}) and observe there is a natural lift of these maps to $\mathbf{R}\mathcal{H}om_{p_X}(\EE,\EE)$. \\
Furthermore, we will have a closer look at the trace-identity split maps 
\begin{equation}
\begin{tikzcd}
\mathbf{R}\Homm_{p_{X}}(\EE,\EE) \arrow[r] \arrow[d] & \mathbf{R}p_{S_*}K_S[-1]\\
\mathbf{R}p_{S_*}\mathcal{O}_{S}.
\end{tikzcd}
\end{equation}
Here, by a split map we mean that the depicted arrows have a right-inverse making $\mathbf{R}p_{S,*}\mathcal{O}_{S}$ and $\mathbf{R}p_{S,*}K_{S}[-1]$ into direct summands of $\mathbf{R}\mathcal{H}om_{p_X}(\EE,\EE)$. We will observe that the \textit{natural} lifts of $\sigma$, $\lambda$ to $\mathbf{R}\mathcal{H}om_{p_X}(\EE,\EE)$ act as the identity on these two terms.
\subsection{Setup}
Let $\langle \iota \rangle $ act on $\N$, where $\iota$ is seen as a $\mathbf{Z}/(2\mathbf{Z})$-action. 
\begin{dfn}\label{equivdef}
	As introduced in Def. \ref{equi}, we call $\At_{\EE}$ $\iota$-equivariant if there is a commutative diagram 
	\begin{equation} 
	\begin{tikzcd}
	[column sep=14ex]
	\mathbf{R}\Homm_{p_X}(\EE,\EE)[1] \arrow[r,"\theta_\iota"] &\mathbf{R}\Homm_{p_X}(\iota^*\EE,\iota^*\EE)[1]\\
	\mathbf{T}_{\N}\arrow[u,"\At_{\EE}"] \arrow[r,"\iota_*"] & \iota^*\mathbf{T}_{\N} \arrow[u,"\iota^*\At_{\EE}"]
	\end{tikzcd}
	\end{equation}	
	%\begin{enumerate}
	%	\item there exists a map $\theta_g: 	\mathbf{R}\Homm_{p_X}(\EE,\EE)[1]  \rightarrow \mathbf{R}\Homm_{p_X}(g^*\EE,g^*\EE)[1]$ such that $g^*\theta_g\circ \theta_g=\id$
	%	\item $\theta_{g}$ commutes with the natural map $g_*: \mathbf{T}_\N \rightarrow g^*\mathbf{T}_\N$ via $\At_{\EE}$. 
	%\end{enumerate} 
	such that the two  triangles
	\begin{center}
		\begin{tikzcd}
		\mathbf{T}_\N \arrow[r,"\iota_*"] \arrow[dr,equal] & \iota^*\mathbf{T}_\N \arrow[d,"\iota^*(\iota_*)"]\\
		& (\iota^{*,2})\mathbf{T}_\N \\
		\end{tikzcd}
	\end{center}
	\begin{center}
		\begin{tikzcd}
		\mathbf{R}\mathcal{H}om_{p_X}(\EE,\EE) \arrow[r,"\theta_{\iota}"] \arrow[dr,equal]&
		\mathbf{R}\mathcal{H}om_{p_X}(\iota^*\EE,\iota^*\EE) \arrow[d,"\iota^*\theta_{\iota}"]\\
		&\mathbf{R}\mathcal{H}om_{p_X}(\iota^{2,*}\EE,\iota^{2,*}\EE) \\
		\end{tikzcd}
	\end{center}
	map to each other via $\At_{\EE}$ and $\iota^*\At_{\EE}$, lifting $\At_{\EE}$ to $\mathbf{D}^b(\N)^{\langle \iota \rangle} $
\end{dfn}
\begin{rmk}
The composition of the differential maps $\mathbf{T}_ \N \rightarrow \iota^*\mathbf{T}_\N \rightarrow \iota^{*,2} \mathbf{T}_\N$ gives naturally the identity, so we get the first triangle for free. Thus we are left to construct a map $$\theta_{\iota}: \mathbf{R}\mathcal{H}om_{p_X}(\EE,\EE)  \rightarrow \mathbf{R}\mathcal{H}om_{p_X}(\iota^*\EE,\iota^*\EE)$$ compatible with $\At_{\EE}$ in the above sense. 
\end{rmk}
\subsection{Strategy}
We decompose $\iota=\sigma \circ\lambda$ into two maps and construct $\theta_{\lambda}, \theta_{\sigma}$ separately. \\
The  determinant twist  $\lambda: \N \rightarrow \N$ was on closed points the map $$[\E_\phi ]\mapsto[ \E_\phi\otimes \pi^*\det(\pi_*\E_\phi)^{-1}].$$
The trace shift $\sigma: \N \rightarrow \N $ is $$[\E_\phi] \mapsto[ \sigma_{\tr\phi}^*\E_\phi].$$
We remark that both $\lambda$ and $\sigma$ have square equal to the identity. \\
Furthermore, we showed in Lem. \ref{commutesigmalambda} that
	$\lambda \sigma=\sigma \lambda$ holds. 
%		Indeed, recalling that $\pi_*\Tr_{\tr\Phi}^*=\pi_* $ and $ \Tr_{\tr\Phi}^*\pi^*=\pi^*$, we observe 
%	\begin{center}
%		$\lambda(\Tr_{\tr\Phi}^*\E)=\Tr_{\tr\Phi}^*\E \otimes \pi^*\det(\pi_*\Tr_{\tr\Phi}^*\E)^{-1}=\Tr_{\tr\Phi}^*\E \otimes \pi^*\det(\pi_*\E)^{-1}=\Tr_{\tr\Phi}^*\E \otimes \Tr_{\tr\Phi}^*\pi^*\det(\pi_*\E)^{-1}=\Tr_{\tr\Phi}^*(\lambda \E)$
%	\end{center}
%\end{proof}

\subsection{Trace and determinant}\label{splitting}
As $X \xrightarrow{\pi} S$ is affine, the functor of Rmk. \ref{classifyingmap} $$\EE \mapsto \pi_*\EE$$ preserves flatness and defines a classifying map 
\begin{align*}
	\Pi: \N &\rightarrow \M \\
	[\E_\phi] &\mapsto [\pi_*\E_\phi]
\end{align*}
to the moduli stack $\M$ of sheaves $\pi_*\E$ on $S$. Thus, $$\EE \mapsto \det(\pi_*\EE)$$
gives the determinant map 
\begin{align*}
	\det\pi_*:\N &\rightarrow \mathbf{Pic}(S);\\
	 [\E_\phi] &\mapsto [\det(\pi_*\E_\phi)]
\end{align*}
We construct the universal section $\tr(\Phi)$ as follows: The map $\N \xrightarrow{\Pi} \M$ "forgetting the Higgs field" has over closed points $[E]\in \M$ fibres given by the vector spaces $\Hom(E,E\otimes K_S)$, parametrising Higgs fields $\phi$. In families, the sheaf $ \Homm(\mathsf{E},\mathsf{E}\otimes K_S)$ admits a universal section $\Phi$ over $\M\times S$. Composing with  $\tr: \Homm(\mathsf{E},\mathsf{E}\otimes K_S)\rightarrow K_S\otimes \OO_{\M\times S}$ and pulling back to $\N\times S$ defines a section $\tr(\Phi)$ of $K_S\otimes \OO_{\N\times S}.$ Thus,
 $$(\mathsf{E},\Phi) \mapsto \tr(\Phi) $$ defines a map
 \begin{align*}
 	\tr: \N &\rightarrow \Gamma(K_S);\\
 	[\E_\phi] &\mapsto \tr\phi.
 \end{align*}
We get global trace and determinant maps 
\begin{equation*}
\begin{tikzcd}
\N \arrow[r] \arrow{d}{\det\pi_*} \arrow{r}{\tr}& \Gamma(K_S)\\
\mathbf{Pic}(S)& 
\end{tikzcd}
\end{equation*}
At the level of tangent spaces at a fixed point $[\E_\phi]\in \N$, the differentials are
\begin{equation*}
\begin{tikzcd}
\Ext^1(\E_\phi,\E_\phi)  \arrow[d] \arrow[r] & H^0(K_S)\\
H^1(\OO_S)& 
\end{tikzcd}
\end{equation*}
Globally on $\N$, we want to identify the maps 
\begin{equation}
\begin{tikzcd}
\mathbf{R}\Homm_{p_{X}}(\EE,\EE) \arrow[r] \arrow[d] & \mathbf{R}p_{S_*}K_S[-1] \\
\mathbf{R}p_{S_*}\mathcal{O}_{S} 
\end{tikzcd}
\end{equation}
as the \textit{virtual} differential of the trace and determinant maps.\footnote{We will use for $\det$ a result of \cite{STV} and for $\tr$ one of \cite{TT}.}\\
These maps are constructed as follows:\\
The arrow
$$\mathbf{R}\mathcal{H}om_{p_X}(\EE,\EE) \xrightarrow{\tr\pi_*} \mathbf{R}p_{S,*}\mathcal{O}_{S}$$ is split with right-inverse $\frac{1}{\rk(\pi_*\E) }(\pi^*\_ \cdot \id)  =\frac{1}{2} \pi^* (\_ \cdot \id)$. \\
Replacing the arrows by their duals gives after applying Grothendieck-Verdier duality the split morphism $$ \mathbf{R}\mathcal{H}om_{p_X}(\EE,\EE) \rightarrow \mathbf{R}p_{S,*}K_{S}[-1].$$
We remark that the duality shifts the RHS by $-1$ as $p_X$ is of dimension $3$ and $p_S$ of dimension $2$. We denote the resulting arrows by $a$ and $b$ respectively, now being Serre dual to $\pi^*(\_ \cdot \id)$ and $\tr(\pi_*\_)$, i.e. $a\circ b=2\cdot \id$. 
\begin{lem}
	The composition $$\mathbf{R}p_{S,*}K_{S}[-1] \xrightarrow{b} \mathbf{R}\mathcal{H}om_{p_X}(\EE,\EE)\xrightarrow{\tr(\pi_*)} \mathbf{R}p_{S,*}\mathcal{O}_{S}$$ is zero.
\end{lem}
\begin{proof}
	We observe that this map factors over $$\mathbf{R}\mathcal{H}om_{p_S}(\mathsf{E},\mathsf{E}\otimes K_S)[-1] \xrightarrow{\pi_*\circ \partial} \mathbf{R}\mathcal{H}om_{p_S}(\mathsf{E},\mathsf{E})\xrightarrow{\tr} \mathbf{R}p_{S,*}\mathcal{O}_{S},$$ which is zero, following from \cite[Prop.2.21]{TT} using the triangle stated in Sec. \ref{TTtriangle}.
\end{proof}
\begin{dfn}\label{globalsplitting}
	This results in a global splitting that we denote as 
	$$ \mathbf{R}\mathcal{H}om_{p_X}(\EE,\EE)=\mathbf{R}\mathcal{H}om_{p_X}(\EE,\EE)^\perp \oplus \mathbf{R}p_{S,*}K_{S}[-1]\oplus \mathbf{R}p_{S,*}\mathcal{O}_{S}.$$ 
\end{dfn}

The following paragraph relates this splitting to $\sigma$ and $\lambda$. 
\subsection{ The line bundle twist}\label{naivelinebundle}
Choosing a twisted family $\EE$ over $X\times \N$, we defined in Def. \ref{spectralfamilies} the line bundle $\LLL:=\pi^*\det(\pi_*\EE)^{-1}$
and observed that $\lambda^*\EE \cong \EE \otimes \LLL$, where $\lambda: \N \rightarrow \N$ is the classifying map for $\EE \mapsto  \EE \otimes \LLL $, on closed points acting as $[\E_\phi] \mapsto [\E_\phi \otimes \pi^*\det(\pi_*\E_\phi)]$. Therefore, there are canonical maps $$\mathbf{R}\mathcal{H}om_{p_X}(\EE,\EE)\cong \mathbf{R}\mathcal{H}om_{p_X}(\EE\otimes \LLL,\EE\otimes \LLL),$$
where the  arrow from left to right is $\lambda_*: g \mapsto g\otimes 1 $ with inverse $\lambda^*$ given by  cancelling $\LLL^\vee\otimes \LLL\cong \OO$.
\begin{claim}
	In view of above splitting \ref{globalsplitting}, $\lambda_*$ is diagonal and acts as $+1$ on the second two factors. 
\end{claim}
\begin{proof}
	This is clear for the $\mathbf{R}p_{S,*}\mathcal{O}_{S}$-term. Now we claim that the maps $\lambda_*,\lambda^*$ fix $\mathbf{R}p_{S,*}K_{S}$ because the arrows 
	\begin{equation*}
	\begin{tikzcd}
	\mathbf{R}\Homm_{p_{X}}(\EE,\EE)  & & \mathbf{R}\mathcal{H}om_{p_X}(\EE\otimes \LLL ,\EE\otimes \LLL)\arrow{ll}[swap]{\lambda^*} \\
	&\mathbf{R}p_{S_*}K_S[-1] \arrow{ul}{b} \arrow{ur}[swap]{b}&
	\end{tikzcd}
	\end{equation*}
	are dual to the trace maps 
	\begin{equation*}
	\begin{tikzcd}
	\mathbf{R}\Homm_{p_{X}}(\EE,\EE)  \arrow[dr,"\tr \pi_*"] \arrow[rr,"\lambda_*"] & & \mathbf{R}\mathcal{H}om_{p_X}(\EE\otimes \LLL ,\EE\otimes \LLL) \arrow{dl}[swap]{\tr\pi_*}\\
	&\mathbf{R}p_{S_*}\OO_S&
	\end{tikzcd}
	\end{equation*}
	and the latter commutes because taking traces factors over the natural evaluation $\LLL^\vee\otimes \LLL\cong \OO$. \\
\end{proof}
\subsection{The trace shift}\label{naivesigma}
In Lem. \ref{sigmaequivariance}, we made the classifying map $\sigma: \N \rightarrow \N$, $[\E_\phi] \mapsto [\sigma_{{\tr \phi}}^*\E_\phi]$ into a locally equivariant action on a universal sheaf $\EE$, lifting $\mathbf{R}\mathcal{H}om_{p_X}(\EE,\EE)$ to $ \mathbf{D}^b(\N)^{\langle \sigma \rangle}$ in Cor. \ref{adjoint}.\\
We need the following lemma:
\begin{lem}
	There exists a smooth ambient space for $X\times \N \subset \mathcal{A}$ extending $\sigma$.
\end{lem}
\begin{proof}
	As $X$ is smooth, it is enough by Lem. \ref{equivemb} to find a $\sigma$-linearised very ample line bundle on $\N$. Again, we may consider $\mathscr{L}\otimes \sigma^*\mathscr{L}$, $\sigma$-linearised by swapping the factors.  
\end{proof}
By Rmk. \ref{equivillusie}, this makes $\mathbf{L}_\N$ $\sigma$-linearised. We can show:
\begin{cor}
	The partial Atiyah class $\At_{\EE}$ on $\N$ admits a $\sigma$-linearisation:
\end{cor}	
\begin{proof}
	Together with a $\sigma$-equivariant smooth embedding $X\times \N \subset \mathcal{A}$, Cor. \ref{adjoint} makes the full Atiyah class $\At_{\EE}$ on $X\times \N$ compatible with $\sigma$ (\cite[Cor.4.4]{R}), making $\At_{\EE}$ into an element of $\Ext^1(\EE,\EE\otimes \mathbf{L}_{X\times \N})^{\langle \sigma \rangle}$. The rest is similar to the construction of the partial Atiyah class (Def. \ref{partialAtdef}), but now equivariant: Composing with the (naturally $\sigma$-equivariant) projection $\mathbf{L}_{X\times \N}\rightarrow {p_\N}^*\mathbf{L}_\N$ 
	gives $$\mathbf{R}\Homm(\EE,\EE) \rightarrow p_\N^*\mathbf{L}_{\N}[1]. $$
	Applying equivariant Grothendieck duality to $p_X$ \footnote{For a statement of this duality, see \cite[Thm.2.27]{R}.} , which is a map of dimension $3$ and noting again that the relative canonical $\omega_{p_X}$ is trivial gives a $\sigma$-equivariant partial Atiyah class $$ \At_{\EE, \N}: \mathbf{R}\mathcal{H}om_{p_X}(\EE,\EE)[2] \rightarrow \mathbf{L}_\N.$$
\end{proof}

\begin{cor}\label{sigmalinear}
	After taking duals, the equivariance can be written as 
	\begin{equation*}\label{firststep}
	\begin{tikzcd}
	\mathbf{R}\Homm_{p_X}(\EE,\EE)[1] \arrow[r,"\sigma_*"] & \mathbf{R}\Homm_{p_X}(\sigma^*\EE,\sigma^*\EE)[1]\\
	\mathbf{T}_\N \arrow[r,"\sigma_*"] \arrow[u,"\At_{\EE,\N}"] &	\sigma^*\mathbf{T}_\N \arrow[u,"\sigma^*\At_{\EE,\N}"] \\
	\end{tikzcd}
	\end{equation*} 
	Here, we remark that the upper horizontal arrow is the push-down via $p_X$ of the adjoint action of $\Psi$ from Cor. \ref{adjoint}.
\end{cor}
We continue with the following claim:
\begin{claim}\label{trivialaction}
	$\sigma_*: \mathbf{R}\mathcal{H}om_{p_X}(\EE,\EE) \rightarrow \mathbf{R}\mathcal{H}om_{p_X}(\sigma^*\EE,\sigma^*\EE)$ is diagonal with respect to the splitting of Def. \ref{globalsplitting} and acts again trivially on the second two sumands.
\end{claim}
\begin{proof}
	This is similar to the previous one and reduces to the fact that 
	\begin{equation*}
	\begin{tikzcd}
	\mathbf{R}\Homm_{p_{X}}(\EE,\EE)  \arrow[dr,"\tr \pi_*"]  & & \mathbf{R}\mathcal{H}om_{p_X}(\sigma^*\EE ,\sigma^*\EE) \arrow{ll}[swap]{\sigma_*} \arrow{dl}[swap]{\tr\pi_*}\\
	&\mathbf{R}p_{S_*}\OO_S&
	\end{tikzcd}
	\end{equation*}
	commutes because $\tr(\pi_*(\Psi^{-1}g\Psi))=\tr(\pi_*g)$.\\
	Dually, this is 
	\begin{equation*}
	\begin{tikzcd}
	\mathbf{R}\Homm_{p_{X}}(\EE,\EE) \arrow[rr,"\sigma_*"] & & \mathbf{R}\mathcal{H}om_{p_X}(\sigma^*\EE ,\sigma^*\EE) \\
	&\mathbf{R}p_{S,*}K_{S}[-1] \arrow{ur}{b} \arrow{ul}[swap]{b}&
	\end{tikzcd}
	\end{equation*}

	Then replace again the arrows by their duals.
\end{proof}
\begin{rmk}
	We end the section with the remark that we can write the natural maps $\lambda_*$ and $\sigma_*$ as diagonal maps $(\lambda_*\oplus 1 \oplus 1)$ and $(\sigma_*\oplus 1 \oplus 1)$ under the identification $$ \mathbf{R}\mathcal{H}om_{p_X}(\EE,\EE)\cong\mathbf{R}\mathcal{H}om_{p_X}(\EE,\EE)^\perp \oplus \mathbf{R}p_{S,*}K_{S}[-1]\oplus \mathbf{R}p_{S,*}\mathcal{O}_{S}$$ 
\end{rmk}
\subsection{Goal}
The goal of the next two sections is to motivate and to prove that the \textit{virtual differential} of $\iota$, i.e. the equivariant structure $\theta_{\iota}$ of the functor $$\mathbf{R}\mathcal{H}om_{p_X}(\EE,\EE) \mapsto \mathbf{R}\mathcal{H}om_{p_X}(\iota\EE,\iota\EE) $$
that commutes with the Atiyah class $\At_{\EE, \N}$ acts as $-1$ on $$\mathbf{R}p_{S,*}K_{S}[-1]\oplus \mathbf{R}p_{S,*}\mathcal{O}_{S}.$$
\newpage
\section{The determinant} \label{determinant}
\setcounter{subsection}{-1}
\subsection{Summary}
We define the equivariance for $\lambda$ compatible with $\At_{\EE}$, i.e. we construct a lift $\theta_\lambda: \mathbf{R}\mathcal{H}om_{p_X}(\EE,\EE)\rightarrow \mathbf{R}\mathcal{H}om_{p_X}(\lambda^*\EE,\lambda^*\EE)$ replacing the natural lift $\lambda_*$ of the previous section. We drop the subscript $\phi$ of $\E_\phi$ in this section.\\
As $\lambda([\E])=[\E\otimes \pi^*\det(\pi_*\E)^{-1}]$, we observe that the following diagram commutes: 
\begin{equation*}
\begin{tikzcd}
\N \arrow[r,"\det\pi_*"] \arrow[d,"\lambda"] & \mathbf{Pic}(S) \arrow[d,"\mathcal{L} \mapsto \mathcal{L}^{-1}"]\\
\N \arrow[r,"\det\pi_*"] &  \mathbf{Pic}(S) 
\end{tikzcd}
\end{equation*}
At tangent spaces for a fixed point $[\E]\in \N$, the differentials are given by 
\begin{equation*}
\begin{tikzcd}
\Ext^1(\E,\E) \arrow[r,"\tr\pi_*"] \arrow[d,"(d\lambda)_{[\E]}"] & H^1(\OO_S) \arrow[d,"\alpha\mapsto -\alpha"]\\
\Ext^1(\lambda\E,\lambda\E) \arrow[r,"\tr\pi_*"] &  H^1(\OO_S).
\end{tikzcd}
\end{equation*}
As $\Ext^1(\E,\E) \rightarrow H^1(\OO_S)$ is split, $d\lambda$ acts as $-1$ on $H^1(\OO_S)$ for \textit{all} points $[\E]$ of $\N$. However, recalling $\lambda^*\EE=\EE\otimes \LLL$ and restricting the natural map of Sec. \ref{naivelinebundle}
$$ \lambda_*: \pot \rightarrow \mathbf{R}\Homm_{p_X}(\lambda^*\EE , \lambda^*\EE)$$
to a closed point $[\E]\in \N$, gives in degree $1$ $$\lambda_*:\Ext^1(\E,\E) \rightarrow \Ext^1(\lambda\E,\lambda \E),$$ which is the identity for all $\E$ where $\det(\pi_*\E)\cong \OO_S$.  \\
We replace $\lambda_*$ by the correct map $\theta_{\lambda}$, whose action on $\mathbf{R}\mathcal{H}om_{p_X}(\EE,\EE)$ is now $-1$ on $\mathbf{R}p_{S,*}\mathcal{O}_{S}$, as motivated above. For this, we need the following ingredient: 
\subsection{Deformations of  the determinant}
By \cite[Prop.3.2]{STV}, there is a commutative square
\begin{equation}\label{TOEN}
\begin{tikzcd}
[column sep=14ex]
\RHomm_{p_{X}}(\EE,\EE)[1] \arrow[r,"\tr  \pi_*"] &\mathbf{R}p_{S,*}\OO_ {S}[1]\\
\mathbf{T}_{\N}\arrow[u,"\At_{\EE}"] \arrow[r,"\det_*"]& \det^*\mathbf{T}_{\mathbf{Pic}(S)} \arrow{u}[swap]{\det^*\At_{\det(\pi_*\EE)}}
\end{tikzcd}
\end{equation}
relating deformations of $[\E] \in \N$ with the one of $[\det(\pi_*\E) ]\in \mathbf{Pic}(S)$.
\subsection{The differential of $\lambda$}
\begin{dfn}
	We define $$\theta_\lambda: \pot \rightarrow \mathbf{R}\Homm_{p_X}(\lambda^*\EE , \lambda^*\EE)$$ as $$f\mapsto \lambda_* f\otimes 1 -  \pi^*(\tr(\pi_*f)\cdot \id)\otimes 1 $$
	i.e. it is the natural map $\lambda_*$  of Sec. \ref{naivelinebundle}, corrected by the differential of the determinant, according to Diag. \ref{TOEN} above.
\end{dfn}
\begin{rmk}
	Corresponding to the splitting $$\mathbf{R}\mathcal{H}om_{p_X}(\EE,\EE)\cong \mathbf{R}\mathcal{H}om_{p_X}(\EE,\EE)_0\oplus\mathbf{R}p_{S,*}\mathcal{O}_{S},$$ we may rewrite $$f=f_0\oplus \frac{1}{2}\pi^*(\tr(\pi_*f))\cdot \id$$ and observe $$\theta_{\lambda}: f_0\oplus \frac{1}{2}\pi^*(\tr(\pi_*f))\cdot \id \mapsto \lambda_*f_0\oplus -\frac{1}{2}\pi^*(\tr(\pi_*f))\cdot \id,$$
	so $\theta_\lambda=\lambda_*\oplus (-1)$ is diagonal. 
\end{rmk}
\begin{lem}\label{diagonallambda}
	We have $\lambda^*\theta_\lambda\circ \theta_\lambda=\id$, i.e.
	\begin{center}
		\begin{tikzcd}
		\pot [1]\arrow[r,"\theta_{\lambda}"] \arrow[dr,equal]& \RHomm_{p_X}(\lambda^*\EE ,\lambda^*\EE)[1] \arrow[d,"\lambda^*\theta_\lambda"]\\
		&\RHomm_{p_X}(\lambda^{*,2}\EE,\lambda^{*,2}\EE)[1] \\
		\end{tikzcd}
	\end{center}
	commutes
\end{lem}
\begin{proof}
	As $\lambda^2=\id$, the vertical arrrow maps indeed to  $\mathbf{R}\mathcal{H}om_{p_X}(\EE,\EE)[1]$.\\
	As $\theta_{\lambda}=\lambda_*\oplus (-1)$, we get $\theta_{\lambda}^2=\id$. 
\end{proof}
\subsection{The Atiyah class}
Via $\At_{\EE}$, $\theta_{\lambda}$ relates to the natural differential map $\lambda_*:\mathbf{T}_\N \rightarrow \lambda^*\mathbf{T}_\N$:
\begin{lem}\label{finallinebundle}
	There is a commutative square
	\begin{equation} 
	\begin{tikzcd}
	[column sep=14ex]
	\pot[1] \arrow[r,"\theta_\lambda"] &\mathbf{R}\Homm_{p_X}(\lambda^*\EE,\lambda^*\EE)[1]\\
	\mathbf{T}_{\N}\arrow[u,"\At_{\EE}"] \arrow[r,"\lambda_*"] & \lambda^*\mathbf{T}_{\N} \arrow[u,"\lambda^*\At_{\EE}"]
	\end{tikzcd}
	\end{equation}
\end{lem}
\begin{proof}
	By functoriality of the Atiyah class, following the RHS up gives $\lambda^*\At_{\EE}\circ \lambda_*=\At_{\lambda^*\EE}=\At_{\EE\otimes \LLL}$, so we are left to prove that $\theta_\lambda\circ \At_{\EE}=\At_{\EE\otimes \LLL}$. We compute
	$$\theta_\lambda \circ \At_{\EE}=\At_{\EE}\otimes 1 - \pi^*(\tr(\At_{\pi_*\EE})\cdot \id)\otimes 1$$
	and $$\At_{\EE\otimes\LLL}=\At_{\EE}\otimes 1 + \At_{\LLL} \cdot \id \otimes 1.$$ Now by Diag. \ref{TOEN}, relating $\det$ and $\tr$ we have
	$$-\pi^*(\tr(\At_{\pi_*\EE}))=\pi^*\At_{\det(\pi_*\EE)^{-1}}=\At_{\pi^*\det(\pi_*\EE)^{-1}}=\At_\LLL$$	
	and the claim follows. 	
\end{proof}
\subsection{The equivariance}
We can now prove that $\At_{\EE}$ is compatible with above constructed map $\theta_{\lambda}$.
\begin{cor}
	$\At_{\EE}$ is $\lambda$-equivariant in the sense of Def. \ref{equivdef}. 
\end{cor}
\begin{proof}
	We see that the commutative triangle
	\begin{center}
		\begin{tikzcd}
		\mathbf{T}_\N \arrow[r,"\lambda_*"] \arrow[dr,equal] & \lambda^*\mathbf{T}_\N \arrow[d,"\lambda^*(\lambda_*)"]\\
		& (\lambda^{*,2})\mathbf{T}_\N \\
		\end{tikzcd}
	\end{center}
	maps via $\At_{\EE}$ to the triangle
	\begin{center}
		\begin{tikzcd}
		\pot [1]\arrow[r,"\theta_{\lambda}"] \arrow[dr,equal]& \RHomm_{p_X}(\lambda^*\EE,\lambda^*\EE)[1] \arrow[d,"\lambda^*\theta_\lambda"]\\
		&\RHomm_{p_X}(\lambda^{*2}\EE,\lambda^{*2}\EE)[1] \\
		\end{tikzcd}
	\end{center}
	Indeed, the second triangle commutes by  Lem. \ref{diagonallambda} and the compatibility between both triangles via $\At_{\EE}$ follows from Lem. \ref{finallinebundle}. This makes $\At_{\EE}$ $\lambda$-equivariant in the sense of Def. $\ref{equivdef}$.
\end{proof}
 
\begin{rmk}
	We see that $\lambda([\E])=[\E]$ on $\N^\perp$, as $\det(\pi_*\E)\cong \OO_S$. Hence by Lem. \ref{diagonallambda}, $\theta_\lambda=1\oplus(-1)$ acts genuinely on $\mathbf{R}\Homm_{p_X}(\EE,\EE)[1]|_{\N^\perp}$, giving a splitting into $\pm 1$ eigensheaves $$\mathbf{R}\Homm_{p_X}(\EE,\EE)[1]_0|_{\N^\perp}\oplus \mathbf{R}p_{S,*}\OO_S[1]|_{\N^\perp}.$$
\end{rmk}

\section{The trace}\label{trace}
\setcounter{subsection}{-1}
\subsection{Summary}
Having dealt with the determinant, we now lift $\sigma $ to $\mathbf{R}\mathcal{H}om_{p_X}(\EE,\EE)$:\\
Recalling Def. \ref{spectralfamilies}, $\sigma:\N \rightarrow \N$ was on closed points acting as $$\sigma([\E_\phi]) =[\sigma_{{\tr \phi}}^*\E_\phi]=[\E_{\phi^\mathfrak{t}}].$$ Thus, there is a commutative square
\begin{equation*}
\begin{tikzcd}
\N \arrow[r,"\tr"] \arrow[d,"\sigma"] & \Gamma(K_S) \arrow[d,"\alpha \mapsto -\alpha"]\\
\N \arrow[r,"\tr"] &  \Gamma(K_S),
\end{tikzcd}
\end{equation*}
because $\tr(\phi^\mathfrak{t})=-\tr(\phi)$. At the level of tangent spaces for a fixed class $[\E_\phi] \in \N$, this gives the diagram
\begin{equation}\label{observation}
\begin{tikzcd}
\Ext^1(\E_{\phi},\E_{\phi}) \arrow[r,"d\tr"] \arrow[d,"(d\sigma_{[\E_\phi]}"] & H^0(K_S) \arrow[d,"\alpha \mapsto -\alpha"]\\
\Ext^1(\sigma_{{\tr \phi}}^*\E_\phi,\sigma_{{\tr \phi}}^*\E_\phi) \arrow[r,"d\tr"] &  H^0(K_S).
\end{tikzcd}
\end{equation}
As we have a splitting $$\Ext^1(\E_\phi,\E_\phi) \cong \Ext^1(\E_{\phi},\E_{\phi})^0\oplus H^0(K_S),$$ $d\sigma_{\tr\Phi}$ acts $-1$ on $H^0(K_S)$ for all points $[\E_{\phi}]$ of $\N$. \\
However, we observe that the natural map induced by pullback
$$\sigma_{{\tr \phi},*}: \Ext^1(\E_\phi,\E_\phi) \rightarrow \Ext^1(\sigma_{{\tr \phi}}^*\E_\phi,\sigma_{{\tr \phi}}^*\E_\phi) $$
acts trivially whenever $\E_\phi$ has centre of mass zero (i.e. $\tr(\phi)=0$). \\
We conclude that the construction of the linearisation (Cor. \ref{adjoint}) $$\sigma_{*}: \mathbf{R}\mathcal{H}om_{p_X}(\EE,\EE) \cong \mathbf{R}\mathcal{H}om_{p_X}(\sigma^*\EE,\sigma^*\EE),$$
commuting with $\At_{\EE}: \mathbf{T}_\N \rightarrow  \mathbf{R}\mathcal{H}om_{p_X}(\EE,\EE)$ needs to be modified.\\
More precisely, we construct the correct map $\theta_{\sigma}$, whose action on $\mathbf{R}\mathcal{H}om_{p_X}(\EE,\EE)[1]$ will be $-1$ on the summand $\mathbf{R}p_{S,*}K_{S}$, generalising Diag. \ref{observation} to families.
\begin{rmk}\label{tangentaffine}
	As $\Gamma(K_S)$ is an affine space, we may identify $$\mathbf{T}_{\Gamma(K_S)}\cong H^0(K_S)\otimes \OO_{\Gamma(K_S)},$$  and we see $\mathbf{T}_{\Gamma(_S), \alpha} \cong H^0(K_S)$ for each point $\alpha \in  \Gamma(K_S)$. 
\end{rmk}

\subsection{First step}
We recall the diagram stated in Cor. \ref{sigmalinear} that we need to modify:
\begin{equation*}
\begin{tikzcd}
\mathbf{R}\Homm_{p_X}(\EE,\EE)[1] \arrow[r,"\sigma_*"] & \mathbf{R}\Homm_{p_X}(\sigma^*\EE,\sigma^*\EE)[1]\\
\mathbf{T}_\N \arrow[r,"\sigma_*"] \arrow[u,"\At_{\EE}"] &	\sigma^*\mathbf{T}_\N \arrow[u,"\sigma^*\At_{\EE}"] \\
\end{tikzcd}.
\end{equation*} 
	The right adaption is writing the map $\sigma$ as a composition $$ \N \xrightarrow{(\tr,\id) } {\Gamma(K_S)}\times \N \xrightarrow{translation } \N$$ such that its differential is $$\tilde{\sigma}_*(v)=\sigma_{*}(v)+\tr_*(v)\cdot \id,$$
	so it contains a correction term differentiating the trace.
	This will be established in the next step. In the third step, we connect this map to the (virtual) differential of the trace $\mathbf{R}\mathcal{H}om_{p_X}(\EE,\EE)[1]\rightarrow \mathbf{R}p_{S,*}K_{S}$ and construct the correct lift $\theta_{\sigma}$ to $\mathbf{R}\mathcal{H}om_{p_X}(\EE,\EE)[1]$ that acts as $-1$ on $ \mathbf{R}p_{S,*}K_{S}.$
\subsection{Second step}
We review the construction of the trace shift $\sigma$ on spectral sheaves from Def. $\ref{spectralfamilies}$: We will make more explicit that this map factors over $\Gamma(K_S)$. More precisely, deformations of $\tr\phi \in \Gamma(K_S)$ induce deformations of $[\E_{\phi}] \in \N$ in the following sense:\\ 
We recall that any section $\alpha \in \Gamma(K_S)$ acts on a Higgs pair $(E,\phi)$ via "translating $\phi$ by $\alpha$", i.e.  the map $$ (E,\phi) \mapsto (E,\phi-\alpha \cdot\id)$$
induces the translation
$$\sigma_{\Gamma(K_S)}: \Gamma(K_S)\times\N \rightarrow\N.$$
In terms of their corresponding spectral sheaves $\E_\phi$ on $X$, this is expressed via the pullback $$(\alpha,[\E_\phi]) \mapsto [\sigma_\alpha^*\E_\phi]=[\E_{\phi-\alpha \cdot \id}],$$ where $\sigma_{\alpha}: X\rightarrow X$ is on local coordinates $(s,t) \in X$ translation on the fibres by $\alpha$,  $$(s,t)\mapsto (s,t-\alpha_s).$$
Thus we may rewrite $\sigma$ as
\begin{equation*} \label{deftr}
\sigma: \N \xrightarrow{\tr \times \id} \Gamma(K_S)\times \N \xrightarrow{\sigma_{\Gamma(K_S)}} \N
\end{equation*}
sending $$[\E_\phi ]\mapsto (\tr\phi,[\E_\phi]) \mapsto [\sigma_{\tr\phi}^*\E_\phi]=[\E_{\phi^\mathfrak{t}}]$$ This agrees with the previous definition on points, but now factors over $\Gamma(K_S)$, inducing differential maps
\begin{equation}\label{difftrace}
\tilde{\sigma}: \mathbf{T}_\N \rightarrow (\tr \times \id)^*[\mathbf{T}_{\Gamma(K_S)}\oplus \mathbf{T}_\N] \rightarrow \sigma^*\mathbf{T}_\N
\end{equation}
and we claim that 
\begin{claim}
	$\tilde{\sigma}_*(v)=\sigma_*(v)+\tr_*(v)\cdot \id$
\end{claim}
\begin{proof}
	Using Rmk. \ref{tangentaffine}, we may write over $\Gamma(K_S)\times \N$ the differential of $\sigma_{\Gamma(K_S)}$ as
	$$\sigma_{\Gamma(K_S),*}: \mathbf{T}_{\Gamma(K_S)} \oplus \mathbf{T}_\N \rightarrow \sigma_{\Gamma(K_S)}^*\mathbf{T}_\N; \hspace{5pt } (\alpha,v) \mapsto \sigma_{\Gamma(K_S),*}(v-\alpha \cdot \id).$$
	As $\tr^*\mathbf{T}_{\Gamma(K_S)}\cong H^0(K_S)\otimes \OO$, pulling back by $(\tr \times \id)$ gives the second arrow in Equ. \ref{difftrace}, since $(\tr \times \id)^*\sigma_{\Gamma(K_S)}^*=\sigma^*$. So pre-composing with $$\mathbf{T}_\N \rightarrow (\tr \times \id)^*[\mathbf{T}_{\Gamma(K_S)}\oplus \mathbf{T}_\N]; \hspace{5pt }v \mapsto (\tr_*v,v)$$
	gives $$ v \mapsto \sigma_{*}(v-\tr_*(v)\cdot \id)= \sigma_{*}(v)+\tr_*(v)\cdot \id$$
	as $ \sigma_{*}\tr_*=-\tr_*$. 
\end{proof}
\begin{lem}\label{helplemma}
	$\tilde{\sigma}_*$ again has square equal to identity on $\mathbf{T}_\N$. Furthermore, we have $\tilde{\sigma}_*\circ \lambda_*=\lambda_*\circ \tilde{\sigma}_*$.
\end{lem}
\begin{proof}
	In fact, we have $\sigma_*^2=\id$. Then note $\sigma_{*}\circ\tr_*=\tr_*\circ \sigma_{*}= -\tr_*.$\\
	The second statement follows from $\sigma\lambda=\lambda\sigma$ ( see Lem. \ref{commutesigmalambda}).
\end{proof}
\begin{dfn}
	Restricting the second arrow in Equ. \ref{difftrace} to the first factor gives $$ \tr^*\mathbf{T}_{\Gamma (K_S)}\rightarrow \mathbf{T}_\N \xrightarrow{\sigma_*} \sigma^*\mathbf{T}_\N,$$
	where we label the first arrow as $\id_*: \tr^*\mathbf{T}_{\Gamma (K_S)}\rightarrow \mathbf{T}_\N$, given by
	$$\id_*(v) = v\cdot \id.$$
\end{dfn}

\subsection{Third Step}
We will use a single result of \cite{TT} at this point:
\begin{claim}
	The differential of $\tr: \N \rightarrow \Gamma(K_S)$
	\begin{center}
		\begin{tikzcd}
		\mathbf{T}_\N \arrow[r,"\tr_*"] &\tr^*\mathbf{T}_{\Gamma (K_S)},
		\end{tikzcd}
	\end{center}
 admits a right-inverse
		\begin{align*}
			\tr^*\mathbf{T}_{\Gamma (K_S)} \xrightarrow{\frac{1}{2}\cdot \id_*} \mathbf{T}_\N,
		\end{align*}
	which are compatible with the arrow $a: \mathbf{R}\mathcal{H}om_{p_X}(\EE,\EE)[1] \rightarrow \mathbf{R}p_{S,*}K_{S}$ and its inverse $b$ via Atiyah classes, i.e. 
	\begin{equation}\label{TT}
	\begin{tikzcd}
	\mathbf{R}\Homm_{p_X}(\EE,\EE)[1] \arrow[r,"a"]&  \mathbf{R}{p_{S,*}}K_S  \arrow[r,"b"] & \mathbf{R}\Homm_{p_X}(\EE,\EE)[1]\\
	\mathbf{T}_{\N} \arrow[r,"\tr_*"] \arrow[u,"\At_{\EE}"] & \tr^*\mathbf{T}_{\Gamma(K_S)} \arrow[u,"\At_{\Gamma(K_S)}"]\arrow[r," \id_*"]& \mathbf{T}_\N \arrow[u,"\At_{\EE}"]\\
	\end{tikzcd}
	\end{equation}
	commutes, following from \cite[Lem.5.28]{TT}. 
\end{claim}	
\subsection{Atiyah classes} 
Composing Diag. \ref{TT} with Cor. \ref{firststep}  gives
\begin{equation}\label{traceAtiyah}
\begin{tikzcd}
\mathbf{R}\Homm_{p_X}(\EE,\EE)[1] \arrow[r,"a"]&  \mathbf{R}{p_{S,*}}K_S  \arrow[r,"b"] & \mathbf{R}\Homm_{p_X}(\sigma^{*}\EE,\sigma^*\EE)[1]\\
\mathbf{T}_{\N} \arrow[r,"\tr_*"] \arrow[u,"\At_{\EE}"] & \tr^*\mathbf{T}_{\Gamma(K_S)} \arrow[u,"\At_{\Gamma(K_S)}"]\arrow[r," \sigma_*\circ \id_*"]& \sigma^*\mathbf{T}_\N \arrow[u,"\sigma^*\At_{\EE}"]\\
\end{tikzcd}
\end{equation}
as $\sigma_{*}b=b$ by Claim \ref{trivialaction}. 
\begin{rmk}
	We remark that the lower horizontal composition is 
	$$v \mapsto \sigma_*(\tr_*(v)\cdot \id)=-\tr_*(v)\cdot \id. $$
\end{rmk}	
\subsection{The differential of $\sigma$ } 
\begin{dfn}
	We define 
	$$\theta_\sigma: \mathbf{R}\Homm_{p_X}(\EE,\EE)[1] \rightarrow \mathbf{R}\Homm_{p_X}(\sigma^*\EE,\sigma^*\EE)[1]$$  $$f \mapsto \sigma_*f-ba(f)$$
\end{dfn}
\begin{rmk}
	Corresponding to $$\mathbf{R}\mathcal{H}om_{p_X}(\EE,\EE)[1] \cong \mathbf{R}\mathcal{H}om_{p_X}(\EE,\EE)^0[1]  \oplus\mathbf{R}p_{S,*}K_{S},$$ we write $f=f^0\oplus \frac{1}{2}ba(f)$, so this can be rewritten as $$\theta_{\sigma}: f^0\oplus \frac{1}{2} ba(f) \mapsto \sigma_*f^0\oplus -\frac{1}{2}ba(f),$$
	i.e. it is $-1$ on second summand. 
\end{rmk}
\begin{lem}\label{diagonaltrace}
	We have $\sigma^*\theta_{\sigma} \circ \theta_{\sigma}=\id$, i.e. 
	\begin{center}
		\begin{tikzcd}
		\pot [1]\arrow[r,"\theta_{\sigma}"] \arrow[dr,equal]& \RHomm_{p_X}(\sigma^*\EE,\sigma^*\EE)[1] \arrow[d,"\sigma^*\theta_\sigma"]\\
		&\RHomm_{p_X}(\sigma^{*2}\EE,\sigma^{*2}\EE)[1]
		\end{tikzcd}
	\end{center}
	commutes
\end{lem}
\begin{proof}
	This is similar to the case of $\lambda$: As $\sigma^2=\id$, the vertical arrow maps indeed to  $\mathbf{R}\mathcal{H}om_{p_X}(\EE,\EE)[1]$.\\
	Now $\theta_{\sigma}=\sigma_*\oplus (-1)$ for the above splitting by the remark, thus $\theta_\sigma^2=\id$. 
\end{proof}

\subsection{The Atiyah class}
Again, we relate $\theta_{\sigma}$ to the differential action on $\mathbf{T}_\N$.
\begin{lem}\label{equivtraceat}
	$\theta_{\sigma}=\sigma_{*}-ba$ commutes with $ \tilde{\sigma}_*: \mathbf{T}_\N \rightarrow \sigma^*\mathbf{T}_\N$ constructed in \ref{deftr}, i.e. the following diagram commutes:
	\begin{equation} \label{finaltr}
	\begin{tikzcd}[column sep=12ex]
	\mathbf{R}\Homm_{p_X}(\EE,\EE)[1] \arrow[r,"\theta_{\sigma}"] & \mathbf{R}\Homm_{p_X}(\sigma^*\EE,\sigma^*\EE)[1]\\
	\mathbf{T}_\N \arrow[r,"\tilde{\sigma}_*"] \arrow[u,"\At_{\EE}"] &	\sigma^*\mathbf{T}_\N \arrow[u,"\sigma^*\At_{\EE}"] \\
	\end{tikzcd}
	\end{equation}
	\begin{proof}
		We recall the diagram discussed in \ref{firststep}: 
		\begin{equation}
		\begin{tikzcd}
		\mathbf{R}\Homm_{p_X}(\EE,\EE)[1] \arrow[r,"\sigma_*"] & \mathbf{R}\Homm_{p_X}(\sigma^*\EE,\sigma^*\EE)[1]\\
		\mathbf{T}_\N \arrow[r,"\sigma_*"] \arrow[u,"\At_{\EE}"] &	\sigma^*\mathbf{T}_\N \arrow[u,"\sigma^*\At_{\EE}"] \\
		\end{tikzcd}
		\end{equation} 
		The claim follows by subtracting the rows of \ref{traceAtiyah}, as we then get that  $$\tilde{\sigma}_*=\sigma_*+\tr_*\_\cdot \id = \sigma_{*} -\sigma_{*}\circ \tr_* \_ \cdot \id \; \; \text{maps to } \; \; \theta_{\sigma}=\sigma_*-ba$$ via $\At_{\EE}$ and its pullback by $\sigma^*$.  
	\end{proof}
\end{lem}

\begin{cor}
	This lifts $\At_{\EE}$ to $\mathbf{D}^b(\N)^{\langle \sigma \rangle}$. 
\end{cor}
\begin{proof}
	As $\tilde{\sigma}_*$ again has square equal to $\id$, since
	% Indeed, using $\tr_*\sigma_*=-\tr_*$ and $\sigma_*\tr_*(\alpha \cdot \id)=-\tr_*(\alpha \cdot \id)$ we observe that $$\tilde{\sigma}_*^2(v)=\tilde{\sigma}_*({\sigma}_*(v)+\tr_*(v)\cdot \id)=v$$ as $\sigma_{*}^2=\id$. This gives
	\begin{center}
		\begin{tikzcd}
		\mathbf{T}_\N \arrow[r,"\tilde{\sigma}_{*}"] \arrow[dr,equal] & \sigma^*\mathbf{T}_\N \arrow[d,"\sigma^*\tilde{\sigma}_*"]\\
		& (\sigma^{*,2})\mathbf{T}_\N \\
		\end{tikzcd}
	\end{center}
	maps via $\At_\EE$ to 
	\begin{center}
		\begin{tikzcd}
		\pot [1]\arrow[r,"\theta_{\sigma}"] \arrow[dr,equal]& \RHomm_{p_X}(\sigma^*\EE,\sigma^*\EE)[1] \arrow[d,"\sigma^*\theta_\sigma"]\\
		&\RHomm_{p_X}(\sigma^{*2}\EE,\sigma^{*2}\EE)[1]
		\end{tikzcd}
	\end{center}
	with everything commutative thanks to Lem. \ref{diagonaltrace}, \ref{equivtraceat} and \ref{helplemma}.
	\begin{rmk}
		Restricting $\theta_\sigma$ to $\N^\perp$, we see that the translation by $\tr(\phi)$ of spectral sheaves $\sigma_{\tr\phi}: \E \mapsto \sigma_{\tr\phi}^*\E $ is trivial as $\tr(\phi)=0$ on $\N^\perp$. Thus $\theta_{\sigma}=1\oplus (-1)$  acts entirely on $ \mathbf{R}\Homm_{p_X}(\EE,\EE)[1]|_{\N^\perp}$ giving a splitting into $\pm 1$ eigensheaves $$\mathbf{R}\Homm_{p_X}(\EE,\EE)^0[1]|_{\N^\perp}\oplus \mathbf{R}p_{S,*}K_{S} |_{\N^\perp}.$$
	\end{rmk}
\end{proof}
\section{The equivariance}\label{iotaequiv}
\subsection{Summary}
We combine the results of the previous two sections.  \\
We've seen that there are linearisation maps $\theta_{\lambda},\theta_{\sigma}$ lifting the actions of $\lambda,\sigma$ on $\N$ to the virtual tangent complex $\mathbf{R}\mathcal{H}om_{p_X}(\EE,\EE)[1]$. This was done in such a way that $\theta_{\lambda},\theta_{\sigma}$ are compatible with the actions $\lambda_*: \mathbf{T}_\N \rightarrow \lambda^*\mathbf{T}_\N $ and $\tilde{\sigma}_*: \mathbf{T}_\N \rightarrow \sigma^*\mathbf{T}_\N$ via $\At_{\EE}$.\\
\subsection{The equivariance of $\iota$}
\begin{dfn}
	Recalling that $\iota=\lambda \circ \sigma$, we define $$\theta_{\iota}:= \theta_{\lambda}\circ \theta_{\sigma}.$$
\end{dfn}
The equivariance of $\iota$ is now a simple corollary: 
\begin{cor}
	We have a linearisation $$\theta_\iota: \mathbf{R}\mathcal{H}om_{p_X}(\EE,\EE)[1]\rightarrow \mathbf{R}\mathcal{H}om_{p_X}(\iota^*\EE,\iota^*\EE)[1]$$ for $\iota$ 
	that acts as $-1$ on $\mathbf{R}p_{S,*}K_{S}\oplus \mathbf{R}p_{S,*}\mathcal{O}_{S}[1]$
\end{cor}
\begin{proof}
	Note that $\lambda\circ \sigma=\sigma \circ \lambda$.
	Again split $$\mathbf{R}\mathcal{H}om_{p_X}(\EE,\EE)[1]=\mathbf{R}\mathcal{H}om_{p_X}(\EE,\EE)[1]^\perp \oplus \mathbf{R}p_{S,*}K_{S}\oplus \mathbf{R}p_{S,*}\mathcal{O}_{S}[1],$$ then we have seen that we can write $\lambda_*,\sigma_{*}$ as $\lambda_*\oplus 1 \oplus 1$ and $\sigma_*\oplus 1 \oplus 1$.\\
	By the previous results of Lem. \ref{diagonallambda} and \ref{diagonaltrace}, we have $$\theta_{\lambda}=\lambda_*\oplus 1 \oplus (-1), \; \theta_{\sigma}=\sigma_*\oplus (-1) \oplus 1$$
	Thus, $$\theta_\iota=\theta_{\sigma} \circ \theta_{\lambda}=(\sigma_*\lambda_*)\oplus (-1) \oplus (-1)=\iota_*\oplus (-1) \oplus (-1)$$
	which also shows that $\iota^*\theta_{\iota}\circ \theta_{\iota}=\id$. 
\end{proof}	
\begin{dfn}
	Furthermore, we define $ \iota_*=(\tilde{\sigma}_*\circ \lambda_*): \mathbf{T}_\N \rightarrow \iota^*\mathbf{T}_\N$. 
\end{dfn}
\begin{rmk}
	By Lem. \ref{helplemma}, we find that $\iota_*^2=\id$ holds.
\end{rmk}	
\begin{thm}\label{iotaequivAt}
	$\At_{\EE}: \mathbf{T}_\N \rightarrow \mathbf{R}\mathcal{H}om_{p_X}(\EE,\EE)[1]$ is $\iota$-equivariant in the sense of Def. \ref{equivdef}.
\end{thm}
\begin{proof}
	Indeed, composing the squares in \ref{finallinebundle} and \ref{finaltr} gives 
	\begin{equation} 
	\begin{tikzcd}[column sep=12ex]
	\mathbf{R}\Homm_{p_X}(\EE,\EE)[1] \arrow[r,"\theta_{\iota}"] & \mathbf{R}\Homm_{p_X}(\iota^*\EE,\iota^*\EE)[1]\\
	\mathbf{T}_\N \arrow[r,"\iota_*"] \arrow[u,"\At_{\EE}"] &	\iota^*\mathbf{T}_\N \arrow[u,"\iota^*\At_{\EE}"] \\
	\end{tikzcd}
	\end{equation}
	such that
	\begin{center}
		\begin{tikzcd}
		\mathbf{T}_\N \arrow[r,"\iota_*"] \arrow[dr,equal] & \iota^*\mathbf{T}_\N \arrow[d,"\iota^*(\iota_*)"]\\
		& (\iota^{*,2})\mathbf{T}_\N 
		\end{tikzcd}
	\end{center}
	maps to the triangle
	\begin{center}
		\begin{tikzcd}
		\RHomm_{p_X}(\EE,\EE)[1]\arrow[r,"\theta_{\iota}"] \arrow[dr,equal]& \RHomm_{p_X}(\iota^*\EE,\iota^*\EE)[1] \arrow[d,"\iota^*\theta_\iota"]\\
		&\RHomm_{p_X}(\iota^{*2}\EE,\iota^{*2}\EE)[1]
		\end{tikzcd}
	\end{center}
	via $\At_{\EE}$ with everything commutative.
\end{proof}
\begin{rmk}\label{movingpart}
	Restricting to $\N^\perp$  gives $\theta_{\iota}=1\oplus (-1)\oplus (-1)$, i.e. a map $$\theta_{\iota}: \RHomm_{p_X}(\EE,\EE)[1]|_{\N^\perp} \rightarrow \RHomm_{p_X}(\EE,\EE)[1]|_{\N^\perp}.$$

	Writing
	$$\RHomm_{p_X}(\EE,\EE)[1]|_{\N^\perp}\cong\RHomm_{p_X}(\EE,\EE)[1]^\perp|_{\N^\perp}\oplus N^{vir}$$
	the second summand $$N^{vir}:=(\mathbf{R}p_{S,*}\OO_S[1]\oplus \mathbf{R}p_{X,*}K_S)|_{\N^\perp}$$ is the virtual normal sheaf, i.e. the $(-1)$-eigensheaf for the action of $\theta_{\iota}$.\\
\end{rmk}
\begin{rmk}
	For the final chapter we phrase everything again in terms of their duals. Let $\mathbf{R}\mathcal{H}om_{p_X}(\EE,\EE)[2]$ denote the virtual $\iota$-equivariant cotangent bundle and $\mathbf{L}_\N$ the truncated cotangent complex of $\N$. 
\end{rmk}
\section{Application to the localisation formula}\label{finalsec}
In this chapter we find a perfect obstruction theory for $\N^\perp$. 
The proof is an adaptation of \cite[Prop.1]{GP} replacing the $\mathbf{C}^\times$-action by $\iota$. \\
For the $\mathbf{C}^\times$-action, the idea is to split the obstruction bundle $V$ over the fixed locus into weight zero and non-zero part, and then to remove the non-zero part, the virtual conormal bundle to the fixed locus.  \\
We do the same for $\iota$,  where we take away the $(-1)$-part of $\RHomm_{p_X}(\EE,\EE)[2]|_{\N^\perp}$, which are the deformations of trace and determinant, according to Rmk. \ref{movingpart} above.\\
We recall the definition of a perfect obstruction theory again. It consists of  
\begin{enumerate}
	\item A two-term complex of vector bundles $V^\bullet=[V^{-1}\rightarrow V^0] \in \mathbf{D}^{[-1,0]}(\N)$.
	\item A morphism $\psi: V^\bullet \rightarrow \mathbf{L}_{\N}$ in $\mathbf{D}^b(\N)$ to the truncated cotangent complex $\mathbf{L}_\N$ inducing an isomorphism on $h^0$ and a surjection on $h^{-1}$.
\end{enumerate}
In order to take $\iota$-invariants and apply \cite{GP}, we need for $\psi$ a representation by complexes and equivariant structure:
\subsection{Equivariant representation}
We sum up the results of the previous sections: In Prop. \ref{partialAtiyah} we've found a $2$-term representation of vector bundles for $$\At_{\EE}: \tau^{[-1,0]}\mathbf{R}\mathcal{H}om_{p_X}(\EE,\EE)[2]\rightarrow  \mathbf{L}_\N$$
making it into a perfect obstruction theory on $\N$, represented by
\begin{equation}\label{finalpot}
[V^{-1} \rightarrow V^0] \xrightarrow{\psi}  [\mathcal{I}/\mathcal{I}^2\rightarrow \Omega_\mathcal{A}|_\N] \in \mathbf{D}^{[-1,0]}(\N).
\end{equation}
We have studied the involution $\iota$ on $\N$ and identified one component of $\N^\iota$ with the $\mathbf{SU}(2)$-locus $$\N^\perp=\{(E,\phi)| \det(E)\cong\OO_S \; \text{and }\;\tr(\phi)=0 \} \subset \N$$
We constructed a lift $\theta_{\iota}$ of $\iota$ to $\mathbf{R}\mathcal{H}om_{p_X}(\EE,\EE)[1]$, compatible with the differential map $\iota_*: \iota^* \mathbf{L}_\N \rightarrow \mathbf{L}_\N$, lifting $\At_{\EE}$ to $\mathbf{D}^b(\N)^{\langle \iota \rangle}$.\\
Finally, we've seen in the last chapter that the restriction $\theta_{\iota}$ to $\N^\perp$ is $1\oplus (-1)$ for $$\mathbf{R}\mathcal{H}om_{p_X}(\EE,\EE)[2]|_{\N^\perp} =\RHomm_{p_X}(\EE,\EE)[2]^\perp|_{\N^\perp} \oplus N^{vir,\vee}$$ 
\begin{rmk}
	The above presentation $\psi$ of the obstruction theory can be chosen to be $\iota$-equivariant: Indeed, we've seen this for $\mathbf{L}_\N$ by finding an $\iota$-equivariant smooth embedding $\N \subset \mathcal{A}$.\\
	Concerning the obstruction theory  $V^\bullet$ in Equ. \ref{finalpot}, going back to the proof pf Prop. \ref{partialAtiyah}, we may choose a very negative $\iota$-\textit{equivariant} resolution $F^\bullet\rightarrow \mathbf{R}\mathcal{H}om(\EE,\EE)$ which gives commutativity of 
	\begin{equation*} 
	\begin{tikzcd}[column sep=12ex]
	\mathbf {R}\Homm_{p_X}(\iota^*\EE,\iota^*\EE)\arrow[r,"\iota^*\theta_{\iota}"] & \mathbf{R}\Homm_{p_X}(\EE,\EE)\\
	p_{\overline{X},*}(\iota^*F^\bullet)  \arrow[r,dashed] \arrow[u, ] & p_{\overline{X},*}F^\bullet \arrow[u,], \\
	\end{tikzcd}
	\end{equation*}
	after replacing the polarisation $\OO(1)$ with a $\iota$-linearised one \footnote{Fixing such a line bundle  $\OO(1)$ as in \ref{equivemb} on $X\times \N$, we  can resolve equivariantly as $$ \dots \rightarrow F^0= H^0(\EE^\vee\otimes \EE(l))\otimes \OO(-l)\twoheadrightarrow\EE^\vee\otimes \E.$$}.\\
	Then following the proof of Prop. \ref{partialAtiyah}, we end up with a genuine map of complexes $$\iota^*[V^{-1}\rightarrow V^0] \rightarrow [V^{-1}\rightarrow V^0] \in \mathbf{D}^{[-1,0]}(\N)$$ representing the lift of $\iota$ to the obstruction complex $$\iota^* \theta_{\iota}:\tau^{[-1,0]}\mathbf{R}\mathcal{H}om_{p_X}(\iota^*\EE,\iota^*\EE)[2] \rightarrow\tau^{[-1,0]}\mathbf{R}\mathcal{H}om_{p_X}(\EE,\EE)[2]. $$
	By the previous Sec. \ref{iotaequiv}, this maps to 
	$$\iota^*[\mathcal{I}/\mathcal{I}^2\rightarrow \Omega_\mathcal{A}|_\N] \rightarrow [\mathcal{I}/\mathcal{I}^2\rightarrow \Omega_\mathcal{A}|_\N] \in \mathbf{D}^{[-1,0]}(\N)$$
	representing
	$$\iota^*\mathbf{L}_\N \rightarrow \mathbf{L}_\N$$
	equivariantly, via the $\iota$-linearised truncated relative Atiyah class $\At_{\EE}$ of Thm. \ref{iotaequivAt}. 
\end{rmk}
\begin{cor}\label{everythingequiv}
	To sum up, we remark that this together with the $\iota$-equivariance shown in Sec. \ref{iotaequiv} lifts $$[V^{-1} \rightarrow V^0] \xrightarrow{\psi}  [\mathcal{I}/\mathcal{I}^2\rightarrow \Omega_\mathcal{A}|_\N]$$ to $\mathbf{D}^{[-1,0]}(\N)^{\langle \iota \rangle}$, i.e. is a morphism of $\iota$-linearised complexes. 
\end{cor}

\subsection{Application to the $\iota$-linearisation}
Let $V^\bullet \in \mathbf{D}^b(\N)^{\langle \iota \rangle }$ on $\N$ as described above in \ref{everythingequiv}. By definition,the linearisation $\theta_\iota$ acting genuinely on $ V^\bullet|_{\N^\perp}$ makes $V^\bullet|_{\N^\perp}$ into a $2$-term complex of $\OO_{\N^\perp}[\langle \iota \rangle]$-modules \footnote{We remark that over  $\N^\perp$, $\iota$ acts via $\theta_{\iota}|_{\N^\perp}$.}. Thus by Lem. \ref{Glinear}, taking fixed parts  $$V^\bullet|_{\N^\perp} \mapsto V^\bullet|_{\N^\perp}^\iota$$ is exact, so $V^\bullet|_{\N^\perp}^\iota$ is a $2$-term  complex of vector bundles on $\N^\perp$.\\ To end this section, let us remark that the sheaves $h^{-1}(V^\bullet|_{\N^\perp}^\iota)$, $h^0(V^\bullet|_{\N^\perp}^\iota)$ are independent of the choice of $V^\bullet \sim W^\bullet \in \mathbf{D}^b(\N)^{\langle \iota \rangle} $. 
\subsection{The localisation formula}
We will now adapt the construction of \cite{GP} mentioned at the beginning, which defines obstruction theories of $\mathbf{C}^\times$-fixed loci. It is a general fact that taking fixed part of a $\mathbf{C}^\times$-equivariant map is exact. Based on what we have said in Lem. \ref{Glinear}, this also applies to $\langle \iota \rangle$. We conclude:
\begin{cor}\label{lastlemma}
	The 2-term complex of locally frees
	$$V^\bullet|_{\N^\perp} ^{ \iota}$$ represents the $\theta_{\iota}$-fixed part $$ \tau^{[-1,0]}\RHomm_{p_X}(\EE,\EE)^\perp[2]|_{\N^\perp}$$
	computing the cohomology sheaves $\Extt_{p_X}^i(\EE|_{\N^\perp},\EE|_{\N^\perp})^\perp$ for $i=1,2$.  
\end{cor}
\begin{thm}
	There is a map
	$V^\bullet|_{\N^\perp} ^{ \iota}  \rightarrow \mathbf{L}_{\N^\perp}$
	defining a perfect obstruction theory on $\N^\perp$. 
\end{thm}

\begin{proof}
	The desired map will be constructed along the way, as well as the explicit representation of $\mathbf{L}_{\N^\perp}$. 
	We start with choosing a $\iota$-equivariant embedding $\N \subset \mathcal{A}$ as in Lem. \ref{equivemb}, making $$\mathbf{L}_\N=[\mathcal{I}/\mathcal{I}^2\rightarrow \Omega_\mathcal{A}|_\N]$$ equivariant.\\
	As $\iota$ lifts to $\mathcal{A}$, let $\mathcal{A}^\iota=\cup_i \mathcal{A}_i$ be the decomposition into irreducible components of the fixed locus and let $\N_i:=\mathcal{A}_i\cap \N$, such that $\N_i \subset \mathcal{A}_i$ is defined by the ideal sheaf $\mathcal{I}_{\N_i}$.\\
	We have seen in Rmk. \ref{component} that $\N^\perp \subset \N^\iota$ is a component of the fixed locus, as it is open and closed therein. Thus, we may choose the corresponding $i$ such that $\N_i=\N^\perp$. We get the representation of the cotangent complex $$\mathbf{L}_{\N^\perp}=[\mathcal{I}_{\N^\perp}/ \mathcal{I}_{\N^\perp}^2 \rightarrow \Omega_{\mathcal{A}_i}|_{\N^\perp}]$$
	As $\Omega_{\mathcal{A}}$ is locally free, there is a natural isomorphism $\Omega_{\mathcal{A}}|^\iota_{\mathcal{A}_i}\cong \Omega_{\mathcal{A}_i}$. Indeed, over $\N^\perp$ the differential $\iota_*$ acts as an involution $\iota_*$ on $ \Omega_{\mathcal{A}}|_{\mathcal{A}_i}$ with fixed part $ \Omega_{\mathcal{A}_i}$. \\
	This gives $\Omega_{\mathcal{A}}|^\iota_{\N^\perp}\cong \Omega_{\mathcal{A}_i}|_{\N^\perp}$ over $\N^\perp$. We have the following square
	\begin{equation*}
	\begin{tikzcd}
	\Omega_{\mathcal{A}_i}|_{\N^\perp} \arrow[r, ""] & \Omega_{\N^\perp}\\
	\Omega_{\mathcal{A}}|^{\iota}_{\N^\perp} \arrow[u, "\sim"] \arrow[r,""] & \Omega_{\N}|^\iota_{\N^\perp} \arrow[u,]\\	
	\end{tikzcd}
	\end{equation*}
	where the horizontal arrows are induced by the natural projections. Tus, we observe that the RHS arrow is onto.\\
	Let $V^\bullet |_{\N^\perp}$ be the restriction of the obstruction complex and $\psi_\perp: V^\bullet |_{\N^\perp}\rightarrow \mathbf{L}_{\N}|_{\N^\perp}$ the pulled back map. As $\psi$ is $\iota$-equivariant, we can take its fixed part $$\psi_\perp^\iota: V^{\bullet}|^\iota_{\N^\perp} \rightarrow {\mathbf{L}_{\N}}|_{\N^\perp}^\iota.$$ Furthermore, denote by $\delta: \mathbf{L}_\N|_{\N^\perp} \rightarrow \mathbf{L}_{\N^\perp}$ the (naturally $\iota$-equivariant) canonical map. Then we claim that the composition  $$V^\bullet |^\iota _{\N^\perp} \xrightarrow{\psi_\perp^\iota} {\mathbf{L}_{\N}}|_{\N^\perp}^\iota \xrightarrow{\delta^\iota} \mathbf{L}_{\N^\perp}$$ defines a perfect obstruction theory on $\N^\perp$:\\
	We remark again that the LHS is given by the two-term complex of vector bundles $V^\bullet|_{\N^\perp} ^{ \iota}$ by the last Lem. \ref{lastlemma}, showing $(1)$.\\
	We need to check the upper conditions on cohomology in $(2)$, which we check for the two maps separately: This is obvious for the restricted map $\psi_\perp$ and thus for $\psi_\perp^\iota$, as taking $\iota$-invariants is exact by Lem. \ref{Glinear}.\\
	The morphism $\mathbf{L}_\N|_{\N^\perp}^{\iota} \xrightarrow{\delta_\perp^\iota} \mathbf{L}_{\N^\perp}$ can be represented by the following diagram with exact rows. 
	\begin{equation*}
	\begin{tikzcd}
	0 \arrow[r] & \text{ker}(a) \arrow[r] \arrow[d,dashed]& \mathcal{I}_\N/ \mathcal{I}_\N^2 |_{\N^\perp}^\iota \arrow[d,"d^{-1}",two heads]\arrow[r,"a"]& \Omega_{\mathcal{A}}|_{\N^\perp}^\iota  \arrow[d,"\sim"] \arrow[r] &\Omega_{\N}|_{\N^\perp}^\iota\ \arrow[d, two heads] \arrow[r]& 0\\
	0 \arrow[r]& \text{ker}(b) \arrow[r]&
	\mathcal{I}_{\N^\perp}/ \mathcal{I}_{\N^\perp}^2 \arrow[r,"b"] & \Omega_{\mathcal{A}_i}|_{\N^\perp} \arrow[r] &\Omega_{\N^\perp}\arrow[r]&0
	\end{tikzcd}
	\end{equation*}
	By what we have already discussed, the right most vertical arrow is onto. As the rows are exact and $d^{-1}$ is onto, we actually get $\Omega_{\N}|_{\N^\perp}^\iota \cong \Omega_{\N^\perp}$. This gives the required property on $h^0$.
	Again as $d^{-1}$ is surjective, so is the induced map $\text{ker}(a) \rightarrow \text{ker}(b)$, i.e. there is an epimorphism on ${h}^{-1}$.  
	This finishes the proof of the theorem.
\end{proof}
\begin{rmk}
	As explained in Sec. \ref{pots}, this endows $\N^\perp$ with a virtual cycle of dimension $0$, see \cite[Def.5.2]{BF}. 
\end{rmk}

\end{document}